\documentclass[12pt]{amsart}

\usepackage{amsmath}
\usepackage{amsfonts}
\usepackage{amssymb}
\usepackage{mathrsfs}
\usepackage{setspace}
\usepackage{color}
\usepackage{graphicx}
\usepackage{epstopdf}
\usepackage{float}
\usepackage{multirow}
\usepackage{mathtools}
\usepackage[lofdepth,lotdepth]{subfig}
\usepackage{graphicx,dblfloatfix}
\usepackage{lipsum} 

\usepackage[round,authoryear,sort]{natbib} 

\usepackage[colorlinks,
            linkcolor=blue,
            anchorcolor=blue,
            citecolor=blue
            ]{hyperref}

\newtheorem{theorem}{Theorem}[section]

\newtheorem{definition}{Definition}[section]
\newtheorem{example}{Example}[section]

\newtheorem{lemma}{Lemma}[section]

\newtheorem{proposition}{Proposition}[section]
\newtheorem{remark}{Remark}[section]

\numberwithin{equation}{section}
\newcommand{\RV}{\mathcal{RV}}

\newcommand{\COM}[1]{}

\newcommand{\norm}[1]{\left\lVert #1 \right\rVert}
\newcommand{\pk}[1]{\mathbb{P} \left\{ #1 \right\} }

\newcommand{\EE}{\mathbb{E}}
\newcommand{\MRV}{\mathcal{MRV}}

\newcommand{\vk}[1]{\boldsymbol{#1}}
\def\X{\vk{X}} 
\def\Y{\vk{Y}}

\def\fracl#1#2{\biggr(\frac{#1}{#2} \biggl) }
\def\e{\mathrm{e}} 

\def\R{\mathbb{R}}

\def\RVEC{\mathrm{RVEC}}
\def\RVLC{\mathrm{RVLC}}

\newcommand{\BQN}{\begin{eqnarray}}
\newcommand{\EQN}{\end{eqnarray}}
\newcommand{\BQNY}{\begin{eqnarray*}}
\newcommand{\EQNY}{\end{eqnarray*}}

\definecolor{c20}{RGB}{255,80,0}
\definecolor{c30}{rgb}{0.,0.,1.}
\definecolor{c40}{rgb}{1,0.1,0.7}
\definecolor{c50}{rgb}{1,0,0}
\definecolor{c60}{rgb}{1,0.9,0.1}

\def\cl#1{\textcolor{c50}{#1}}
\def\cl#1{#1}

\def\kai#1{\textcolor{c20}{#1}}
\def\kai#1{#1}
\RequirePackage{colortbl} 
\RequirePackage{soul} 

\begin{document}
\title{Tail Asymptotic of Heavy-Tail Risks with Elliptical Copula}

\author{Kai Wang}
\address[Kai Wang and Chengxiu Ling]
 {Wisdom Lake Academy of Pharmacy \\
 Xi'an Jiaotong-Liverpool University\\
 Ren'ai Road 111, Suzhou, China, 215123}
 \address[Kai Wang]
 {Department of Mathematical Sciences \\
University of Liverpool\\
 Liverpool, UK, L693BX}
 
\author{Chengxiu Ling$^*$}
\email[Chengxiu Ling]{Chengxiu.Ling@xjtlu.edu.cn}

\date{\today}

\thanks{$^*$Corresponding author. E-mail: Chengxiu.Ling@xjtlu.edu.cn}

\begin{abstract}
We consider a family of multivariate distributions with heavy-tailed margins and the type I elliptical dependence structure. This class of risks is common in finance, insurance, environmental and biostatistic applications. We obtain the asymptotic tail risk probabilities and characterize the multivariate regular variation property. The results demonstrate how the rate of decay of probabilities on tail sets varies in tail sets and the covariance matrix of the elliptical copula.  The theoretical results are well illustrated by typical examples and numerical simulations. A real data application shows its advantages in a more flexible dependence structure to characterize joint insurance losses.
\end{abstract}



\maketitle

\baselineskip17pt


\section{Introduction}

Heavy-tail risks characterize the phenomena where huge losses occur with a relatively high probability, such as record-breaking insurance losses, financial asset returns, transmission rates and durations in data networks \citep{Resnick2007, embrechts2013modelling}. For the joint heavy-tail risk events, the copula allows the modelling of the dependence structure with any choice of marginal risks. While the classical Gaussian copula with log-normal and regularly varying risks have been extensively studied in both theoretical and applied contexts \citep{Asmussen2008,Hua2014,Liu_Yang2021,Das2023heavy}, the elliptical copula, as a generalization of the Gaussian case, facilitates more dependence structures and are widely used in the fields of statistics, finance, insurance, environment and biology \citep{FRAHM2003}. For other applications, we refer to risk analysis \citep{Antoine2018,Yin2021}, portfolio theory \citep{LEI2023}, biostatistics \citep{Lindsey1999}, bioinformatics \citep{Dang2015,HE2019}, and climate models \citep{Louie2014}. \cl{However, the probability law of heavy tail risks with general elliptical dependence structure remain unclear. Its exploration will help the simultaneous risk inference for many application-oriented problems.} 

Recall that for an elliptical random vector $\vk Z=(Z_1, \ldots, Z_d) \sim G$ with zero means, it follows by \cite{Fang1990} that (notation: $\stackrel{d}{=}$ stands for equality in distribution)
\begin{eqnarray}\label{Eq.Elliptical}
   {\vk Z} \stackrel{d}{=} R A{\bf U}, 
\end{eqnarray}
where $R>0$ stands for radius, matrix $A \in \R^{d\times d}$ with $AA^\top =:\Sigma$ which is a positive definite covariance matrix, and ${\bf U}$ is a $d$-dimensional random vector uniformly distributed on the unit sphere $S_d = \{{\bf s}\in \mathbb{R}^d: s_1^2 + \cdots + s_d^2 = 1\}$, and ${\bf U}$ is independent of $R$. Further, suppose all margins $G_1, \ldots, G_d$ are continuous, then the elliptical copula of $\vk Z$ is a function $C:[0,1]^d \to [0,1]$ such that
\begin{eqnarray*}
C_{\Sigma,R}\left(u_1, \ldots, u_d\right)=G\left(G_1^{\leftarrow}\left(u_1\right), \ldots, G_d^{\leftarrow}\left(u_d\right)\right), \quad 0\le u_1,\ldots,u_d\le 1,
\end{eqnarray*}
where $\Sigma$ and $R$ are defined in Eq.\eqref{Eq.Elliptical}, and
\begin{eqnarray*}
G_i^{\leftarrow}\left(u\right):=\inf \left\{x: G_i(x)>u\right\}, \quad i=1, \ldots, d,
\end{eqnarray*}
are the marginal left-inverse quantile functions, as well-known concepts like Value-at-Risk (VaR) in financial and risk management \citep{Embrechts2009}. 
\kai{With heavy-tailed marginal risks, a natural question lies in the effect of an elliptical copula and its parameters on the tail behaviour of joint extreme sets. While the case of elliptical copula ($C_{\Sigma,R}$) with the radius $R$ (see Eq. \eqref{Eq.Elliptical}) belonging to the Fr\'echet max domain of attraction (MDA) is discussed in \cite{HASHORVA20061427, peng2008estimating, Li2009goodness,Kluppelberg2017}, it remains a challenge for the consideration of the radius $R$ in the Gumbel MDA \citep[Theorem 1.6.2]{leadbetter2012extremes}, where}
\begin{eqnarray}\label{Eq: def-Gumbel}
    \frac{\pk{R > u+x/w(u)}}{\pk{R>u}} \to \e^{-x},\quad \forall x\in\R
\end{eqnarray}
for positive scaling function $w(\cdot)$ being a regular varying function with index $-\tau,\,\tau\ge-1$, denoted as $w\in\RV_{-\tau}$, that is  \citep{resnick2008extreme},
\begin{eqnarray*}\lim_{t\to\infty}w(tx)/w(t) = x^\tau, \quad x>0.\end{eqnarray*} 
\cl{We denote by $R\in \mbox{\mbox{GMDA}}(w,\tau)$ in this context, and call such an elliptical copula with the radius $R\in \mbox{GMDA}(w, \tau)$ as type I elliptical copula} \citep{Hashorva2007}. One sub-class of type I elliptical family, the elliptical copulas with the Weibullian-type radius \citep{debicki2011,debicki2018extremes} (defined below in Eq.\eqref{R}) includes well-known copulas like Gaussian ($\beta=-1,\gamma=2$), Laplace (\cl{$\gamma =1$}), multivariate power exponential, and Kotz-type, which are studied with wide applications \citep{Song2010,Rosy2021,Masarotto2017}.  The Weibullian-type $R$ exhibits power-exponential tail behaviour, i.e., for large $u$,
\begin{eqnarray}\label{R}
\pk {R>u} \sim Cu^\beta \exp(-Lu^\gamma),
\end{eqnarray}
where \kai{$C,L,\gamma>0$ and $\beta \in \R$.}  Clearly, $R\in \mbox{GMDA}(w,\gamma-1)$ with $w(u) = \gamma L u^{\gamma-1}$.


In this paper, we focus on the heavy-tailed joint risk with regularly varying margins, and with the dependence structure of elliptical copulas ($C_{\Sigma,R}$) whose radius $R$ is of Weibullian-type. Our contribution lies in the asymptotic probability on the tail set, the rate of decay in the probability of joint tail events, and the multivariate regular variation/hidden regular variation \citep{resnick2002hidden}.

The remainder of the paper is organized as follows. Section \ref{sec: preliminary} presents some preliminaries about the elliptical distribution, the tail set, and certain definitions of multivariate regular variation. The main results of tail asymptotics and multivariate regular variation are presented in Section \ref{sec: main}. Then we provide examples of multivariate power exponential copula, Laplace copula and Gaussian copula in Section \ref{sec: Ex} and conduct simulations in Section \ref{sec: Simulation}. Finally, we discuss the applications in Section \ref{sec: app} and conclude this paper in Section \ref{sec: con}. All the proofs are deferred to Section \ref{sec: proof}.

 \section{Preliminaries}\label{sec: preliminary}
\subsection{Notation}
We shall introduce first some standard notation. Let in the following $I, J$ be two non-empty disjoint index sets such that $I \cup J=\mathbb{I}:=\{1, \ldots, d\}, d \geq 2$, and define for $\vk x=\left(x_1, \ldots, x_d\right)^{\top} \in \R^d$ the subvector with respect to the index set $I$ by $\vk x_I:=\left(x_i, i \in I\right)^{\top} \in \R^d$. If $\Sigma \in \R^{d \times d}$ is a square matrix then the matrix $\Sigma_{I J}$ is obtained by deleting both the rows and the columns of $\Sigma$ with indices in $J$ and in $I$, respectively. Similarly we define $\Sigma_{J I}, \Sigma_{J J}, \Sigma_{I I}$. For notational simplicity we write $\boldsymbol{x}_I^{\top}, \Sigma_{J J}^{-1}$ instead of $\left(\boldsymbol{x}_I\right)^{\top},\left(\Sigma_{J J}\right)^{-1}$, respectively. All vector operations are component-wise (given $\vk a, \vk x, \vk y \in \mathbb{R}^d,c \in \R$), and
$$
\begin{aligned}
\boldsymbol{x} & \ge\boldsymbol{y}, \text { if } x_i\ge y_i, \quad \forall i=1, \ldots, d, \\
\boldsymbol{x}+\boldsymbol{y} & :=\left(x_1+y_1, \ldots, x_d+y_d\right)^{\top}, \\
c \boldsymbol{x} & :=\left(c x_1, \ldots, c x_d\right)^{\top}, \quad \boldsymbol{a x} :=\left(a_1 x_1, \ldots, a_d x_d\right)^{\top}\\
\mathbf{0} & :=(0, \ldots, 0)^{\top} \in \mathbb{R}^d, \quad \mathbf{1}:=(1, \ldots, 1)^{\top} \in \mathbb{R}^d, \\
\vk e^* &:= (0, \ldots,1,\ldots, 0)^{\top}\text{ with only one non-zero entry 1 at the $j$-th coordinate},\\
\left\|\boldsymbol{x}_I\right\|^2 & :=\boldsymbol{x}_I^{\top} \Sigma_{I I}^{-1} \boldsymbol{x}_I, \quad \mathcal{S}^{m-1}:=\left\{\boldsymbol{x} \in \mathbb{R}^m: \boldsymbol{x}^{\top} \boldsymbol{x}=1\right\}, \quad m \geq 1 .
\end{aligned}
$$
Throughout this paper, we write $\vk Z \sim G\in \R^d$  for a $d$-dimensional elliptical random vector and ${\vk X} = (X_1, \ldots, X_d) \sim F$ with the same elliptical copula as $\vk Z$, while the marginal distributions $F_1, \ldots, F_d$ are tail equivalent, $\overline F_{j}(t) \sim c_j\overline F_1(t), c_j>0$, and $\overline F_1 \in \mathcal{RV}_{-\alpha}$. We denote $\vk X \in \RVEC(\alpha, b, \Sigma, R)$ with $b(t):=F_1^{\leftarrow}(1-1/t),\, t>1$ such that $b\in \mathcal{RV}_{1/\alpha}$. \cl{All} the limits are taken as $t \to \infty$ unless stated otherwise. 


\newpage
\subsection{Elliptical copula with Weibullian-type radius}
For the elliptical ${\vk Z} \stackrel{d}{=} R A{\bf U}$ with Weibullian-type radius $R$, all the margins $Z_j$ are also of Weibullian type by Lemma \ref{Lemma: RtoZ}. This allows us to determine the tail set in Section \ref{sec: tail set} knowing the tail of either $Z_j$ or $R$.

\subsection{Tail set, threshold and tail probability}\label{sec: tail set}

For the tail exceedance probabilities, we consider the tail set $A_{\vk x_S}$ with thresholds $\vk x_S$ (cf. \citet[Lemma 1]{Das2023heavy}). For a non-empty index set $S \subseteq \mathbb{I} :=\{1, \ldots, d\}$, 
\begin{eqnarray}\label{def: tail set}
   A_{\vk x_S}:=\left\{y \in \mathbb{R}_{+}^d: y_j>x_j, \forall j \in S\right\},
 \end{eqnarray}
with $\vk x_S=(x_j)_{j \in S} >\mathbf0$. Then $\pk{\vk X \in tA_{\vk x_S}} = \pk{X_j > tx_j, j\in S, X_{j'}>0,\, j'\notin S}$,
illustrates the joint exceedance probability with thresholds $tx_j$ on the interested risks $X_j$, for $j \in S$. 

Note that for a heavy-tailed multivariate risk \cl{$\vk X$} possessing the same elliptical copula as $\vk Z \sim G$, the tail behavior of $\vk X$ can be transformed to that for $\vk Z$ \cl{on an alternative tail set}, determined by the marginal tails of both $\X$ and $\vk Z$: 
\begin{eqnarray}\label{Eq: transform}
    \pk{\vk X \in tA_{\vk x_S}} &=& \pk{X_j > tx_j, j\in S, X_{j'}>0,\, j'\notin S} \notag \\
    &=& \pk{F_j(X_j) > F_j(tx_j), j\in S} \notag\\
    &=& \pk{Z_j > G_1^\leftarrow(F_j(tx_j)), j\in S}
    \notag\\     &=& \pk{Z_j > \overline G_1^\leftarrow(\overline F_j(tx_j)), j\in S},
\end{eqnarray}
where the last step is obtained by
the identically distributed marginals of elliptical $\vk Z$. To show Proposition \ref{Proposition: type I} for the approximation of  Eq.\eqref{Eq: transform} in , we introduce first Lemmas \ref{Lemma: threshold} and \ref{Lemma: quadratic} for the approximation of threshold $z_{x,c_j}(t) := \overline G_1^\leftarrow(c_j\overline F_1(tx))$ and the quadratic programming, respectively.

\begin{lemma} \label{Lemma: threshold}
Let $\vk Z$ be an elliptical random vector defined in Lemma \ref{Lemma: RtoZ} with $\overline G_1 (u) = \pk{Z_1>u} \sim C'u^{\beta'}\exp(-Lu^\gamma)$. Suppose that $\overline F_0(t) = \left(t^\alpha \ell(t) \right)^{-1} \in \RV_{-\alpha}$ with $\alpha>0$. Define $z_{x,c}(t) :=\overline G_{1}^\leftarrow(c\overline F_0 (tx))$ for $c>0$. 
Then as $t \rightarrow \infty$,
\begin{eqnarray*}
z_{x,c}(t) &=&  \left(\frac{\alpha\log t}{L}\right)^{1/\gamma} + \frac{\log \left((c^{-1}x^\alpha C' \alpha^{{\beta'}/\gamma} L^{-{\beta'}/\gamma})^{1/\gamma L}\right) }{ \left( \frac{\alpha \log t}{L} \right)^{1-1/\gamma}} \\
&& + \frac{\log \left( [\ell(t) (\log t)^{{{\beta'}/\gamma}}]^{1/\gamma L} \right)}{ \left( \frac{\alpha \log t}{L} \right)^{1-1/\gamma}} +o \fracl{1}{(\log t)^{1-1/\gamma}}.
\end{eqnarray*}
\end{lemma}

Next, we introduce the quadratic programming problem $\mathscr{P}\left(\Sigma^{-1}\right)$ in Lemma \ref{Lemma: quadratic}, \cl{which is widely used to investigate the joint tail behavior of elliptical risks, see}  \citet[Proposition 2.1]{Hashorva2005} and \citet[Proposition 2.5 and Corollary 2.7]{Hashorva2002}. 

\begin{lemma}\label{Lemma: quadratic}
 Let $\Sigma \in \R^{d \times d},\,d\ge 2$ be a positive definite correlation matrix. Then the quadratic programming problem:
\begin{eqnarray*}
\mathscr{P}\left(\Sigma^{-1}\right): \min\limits_{\vk x \geq \vk 1} \vk x^\top \Sigma^{-1} \vk x
\end{eqnarray*}
has a unique solution $\vk e^*\in \R^d$ and there exists a unique non-empty index set $I \subset \{1,\ldots,d\}$ with $J= \{1,\ldots,d\}\backslash I$, so that 
\begin{eqnarray*}
\lambda:=\lambda(\Sigma)=\min_{x \ge \vk{1}}\vk x^\top \Sigma^{-1} \vk x={\vk{e}^*}^\top \Sigma^{-1} \vk{e}^*  ={\vk{e}_I}^\top\Sigma_{II}^{-1} \vk{e}_I\ge1,
\end{eqnarray*}
and if $\Sigma^{-1}\vk 1 \ge \vk 0$, then $\vk e^*=\vk 1$, otherwise the solution is given by
\begin{eqnarray*}
\vk e_I^*=\vk 1_I \text { and } \vk e_J^*  =\Sigma_{J I}\left(\Sigma_{II}\right)^{-1} \vk 1_I\geqslant \vk 1_J.
\end{eqnarray*}
Further for any vector $\vk c \in \R^d$ we have
\begin{eqnarray*}
\vk c^\top \Sigma^{-1} \vk e^*={\vk c_I}^\top\Sigma_{II}^{-1} \vk 1_I.
\end{eqnarray*}
\end{lemma}
\cl{Now, we are ready to derive in Proposition \ref{Proposition: type I}} the approximations of Eq.\eqref{Eq: transform}, i.e.,  the joint tail behavior of $\vk Z$ with Weibullian-type radius \cl{exceeding} the thresholds $z_{x, c_j}(t)$'s.

\begin{proposition}\label{Proposition: type I}
Let $\vk Z$ be an elliptical random vector defined in Lemma \ref{Lemma: RtoZ} 
and $\vk e^*, \lambda, I, J$ be defined w.r.t. $\Sigma=A^\top A$ in $\mathscr{P}\left(\Sigma^{-1}\right)$ in Lemma \ref{Lemma: quadratic}.  Let $\Y$ be a Gaussian random vector in $\R^d$ with covariance
matrix $\Sigma$, $\vk z \in \R^d$. For two measurable functions $\mathcal{L}, u:(0, \infty) \rightarrow(0, \infty)$ satisfying $u(t) \to \infty$ and $\log \mathcal{L}(t) / (u(t))^{\gamma-1} \to 0$ as $t \rightarrow \infty$, we have

\begin{eqnarray}
&&\pk{\vk Z > u(t) \vk{1} + \frac{\vk {z}}{(u(t))^{\gamma-1}} + \frac{\log \mathcal{L}(t)}{(u(t))^{\gamma-1}}\vk{1} +o\fracl{1}{(u(t))^{\gamma-1}}}
=:\pk{\vk Z > u(t) \vk{1} + \vk x^{(t)}} \notag\\
&&\sim \Upsilon \frac{(\lambda u^2(t)+ 2u(t) \vk x^{(t)\top} \Sigma^{-1} \vk e^*)^{\beta_{\gamma,I,J}} }{(u(t))^{|I|}}\exp\left(-L(\lambda u^2(t)+ 2u(t) \vk x^{(t)\top} \Sigma^{-1} \vk e^*)^{\gamma/2}\right)\notag.
\end{eqnarray}
where 
\begin{eqnarray*}
\Upsilon &:=& \Upsilon (\Sigma) = C{(\gamma L)^{1+|J|/2-d}}\frac{\Gamma(d/2) 2^{d/2-1} \pk{\vk Y_J> \widetilde{\vk{u}}_J | \vk Y_I = \vk 0_I}}{(2\pi)^{|I|/2} |\Sigma_{II}|^{1/2} \prod_{i \in I} h_i },\\
h_i&:=& h_i (\Sigma) =\vk {1}_I^{\top}\Sigma_{II}^{-1} \vk e_i, \\
\lambda &:=& \lambda (\Sigma)= \vk 1_I^\top \Sigma_{II}^{-1} \vk 1_I = \min\limits_{\vk x \ge \vk 1}\vk x^\top \Sigma^{-1} \vk x,\\
\beta_{\gamma,I,J} &:=& \beta_{\gamma,I,J} (\Sigma)=\frac{\beta+|I| + \gamma(1+|J|/2-d)}{2}, \ |I|+|J| = d,\\
\widetilde{\vk{u}}_J &:=& \widetilde{\vk{u}}_J (\Sigma) = \lim_{t \to \infty} u(t) \left(\vk{1}_J-\Sigma_{JI}\Sigma^{-1}_{II} \vk{1}_I \right),
\end{eqnarray*}
and set $\pk{\vk Y_J> \widetilde{\vk{u}}_J | \vk Y_I = \vk 0_I} = 1$ if $|I|=d$. The elements in $\widetilde{\vk{u}}_J$ take values in $\{-\infty,0\}$, and $\Sigma_{J I}\left(\Sigma_{II}\right)^{-1} \vk 1_I = \vk e_J^*  \geq \vk 1_J$.
\end{proposition}

\begin{remark}
$(a)$ \kai{This proposition is a modified version of \citet[Theorem 3.1]{Hashorva2007} for Weibullian-type radius $\pk {R>u} \sim Cu^\beta \exp(-Lu^\gamma)$ and the threshold $u(t) \vk{1} + \vk x^{(t)}$. The function $u(t)$ refers to the dominant (first) term in $z_{x, c}(t)$, while $\vk x^{(t)}$ constitutes the remaining terms that behave as constants.}\\
$(b)$ For the multivariate Gaussian random vector $\vk Z$, the result of Proposition \ref{Proposition: type I} coincides with \citet[Proposistion 1]{Das2023heavy}.

\end{remark}

\subsection{Regular variation}
Besides the asymptotic probability on the tail set, we will also characterize the multivariate regular variation property (see Theorem \ref{Thm: MRV}). As all type I elliptical copulas admit asymptotic (tail) independence \citep{FRAHM2003}, we consider the hidden regular variation with extensions to the sub-cones $\EE_d^{(i)}$'s below (cf.  \cite{resnick2002hidden,resnick2008extreme,Das2023heavy}).
\begin{eqnarray*}\EE_d^{(i)} :=\{\vk{y}\in\R_+^d: y_{(i)} >0\}, \  1\le i \le d,\end{eqnarray*} 
where $y_{(1)}\ge y_{(2)}\ge\cdots\ge y_{(i)}\ge \cdots \ge y_{(d)}$ are the associated decreasing ordered statistic of $y_1,\ldots, y_d$. The sub-cone $\EE^{(1)}_d = [0, \infty]^d \backslash \{\vk 0\}$, $\EE^{(2)}_d = \EE^{(1)}_d \backslash \bigcup\limits_{i=1}^{d} \mathbb{L}_{i}$, where $\mathbb{L}_{i}:= \{t\vk e_i, \ t>0\}, 1\le i \le d$, are the axes originating at $\{\vk 0\}$, and 
\begin{eqnarray*}\EE_d^{(d)} \subset \cdots \subset \EE_d^{(2)} \subset \EE_d^{(1)}.\end{eqnarray*}
\begin{definition}[Multivariate regular variation]
Suppose $\vk Y=(Y_1, \ldots, Y_d)$ is a random vector on $[0, \infty)^d$. The distribution of $\vk Y$ is multivariate regularly varying with tail measure $\nu(\cdot)$, if there exists a function $b(t) \to \infty$ as $t \to \infty$ and a non-negative Radon measure $\nu \ne 0$ such that
\begin{eqnarray}\label{MRV}
    t\pk{\frac{\vk Y}{b(t)}\in \cdot} \stackrel{v}\to \nu(\cdot)
\end{eqnarray}
on the cone $\EE^{(1)}_d = [0, \infty]^d \backslash \{\vk 0\}$. Here $\stackrel{v}\to$ stands for vague convergence of measures. 
\\
The above condition implies there exists a constant $\alpha > 0$ such that $b(\cdot) \in \RV_{1/\alpha}$ and the measure $\nu(\cdot)$ is $(-\alpha)$-homogeneous, i.e., for $B \subset \EE^{(1)}_d$
\begin{eqnarray}\label{MRV_nu}
\nu (cB)= c^{-\alpha} \nu(B), \quad c>0.
\end{eqnarray}
We abbreviate the MRV property as $\vk Y \in \MRV(\alpha,b,\nu,\EE_d^{(1)})$.
\end{definition}

The asymptotic independence means the probability of two or more components being simultaneously large with suitable scales $b(\cdot)$ is negligible compared with the probability of one component being large. This also indicates the Radon measure $\nu(\cdot)$ concentrates on the axes.
\cite{resnick2002hidden} considered the fine structures presented in the interior of the cone with relatively crude normalization $b^*(t)$, which is called the hidden regular variation.

\begin{definition}[Hidden regular variation]
The random vector $\vk Y=(Y_1, \ldots, Y_d)$ has a distribution possessing hidden regular variation if, in addition to Eq.\eqref{MRV}, there exists a non-decreasing function $b^*(t) \to \infty$ such that $b(t)/b^*(t) \to \infty$ and there exists a non-negative Radon measure $\nu^* \ne 0$ on $\EE_d^{(2)}$ such that
\begin{eqnarray*}
t\pk{\frac{\vk Y}{b^*(t)}\in \cdot} \stackrel{v}\to \nu^*(\cdot)\end{eqnarray*}
on the cone $\EE_d^{(2)}$. Then there exists $\alpha^* \ge \alpha$ such that $b^* \in \RV_{1/\alpha^*}$, and $\nu^*, \alpha^*$ satisfying the analog of Eq.\eqref{MRV_nu}.
\end{definition}

For the intrinsic dependence structure in the sub-cones $\mathbb{E}_d^{(i)}, i=1,\ldots,d$, we borrow the definition of the multivariate regular variation in \citet[Defninition 3]{Das2023heavy}, followed by a lemma (cf. \citet[Lemma 1]{Das2023heavy})  to characterize such properties with tail set $A_{\vk x_S}$.

\begin{definition}[$\MRV$ on $\EE_d^{(i)}, 1\le i \le d$]
A random vector $\vk Y \in \R^d$ is multivariate regularly varying on $\EE_d^{(i)}$ if there exists a regularly varying function $b_i \in \RV_{1 / \alpha_i}, \alpha_i>0$, and a non-null (Borel) measure $\nu_i$ which is finite on Borel sets bounded away from $\left\{y \in \R_{+}^d: y_{(i)}=0\right\}$ such that
\begin{eqnarray*}
\lim _{t \rightarrow \infty} t \pk{\frac{\vk Y}{b_i(t)} \in A}=\nu_i(A)
\end{eqnarray*}
for all Borel sets $A \in \mathcal{B}\left(\EE_d^{(i)}\right)$ which are bounded away from $\left\{y \in \mathbb{R}_{+}^d: y_{(i)}=0\right\}$ with $\nu_i(\partial A)=0$. We write $\vk Y \in \MRV(\alpha_i, b_i, \mu_i, \mathbb{E}_d^{(i)})$. The limit measure $\nu_i$ obtained is homogeneous of order $-\alpha_i$, i.e., $\nu_i(c A)=c^{-\alpha_i} \nu_i(A)$ for any $c>0$.
\end{definition}

\begin{lemma}[{\citet[Lemma 1]{Das2023heavy}}]\label{Lemma: cone}
    Let $\vk Y$ be a random vector in $\R^d$. Suppose that $b_i \in \RV_{1/\alpha_i}$ for $\alpha_i>0$ is a measurable function and $\nu_i$ is a non-null (Borel) measure on $\mathcal{B}(\EE_d^{(i)})$ which is finite on Borel sets bounded away from $\{\vk y \in \R^d_+: y_{(i)}=0 \}$. Then $\vk Y \in \MRV(\alpha_i,b_i,\nu_i,\EE_d^{(i)})$ if and only if
\begin{eqnarray*}
\lim\limits_{t \to \infty} t\pk{\frac{\vk Y}{b_i(t)}\in A_{\vk x_S}}= \nu_i(A_{\vk x_S})
\end{eqnarray*}
for all sets $S \subset \mathbb{I}$ with $|S|\ge i$, for all $x_j > 0, \forall j \in S$ such that $\vk x_S = (x_j)_{j \in S}$ and
\begin{eqnarray*}
A_{\vk x_S} = \{\vk y \in \R^d_+: y_j > x_j, \forall j \in S \}
\end{eqnarray*}
with $\nu_i\left(\partial A_{\vk x_S}\right)=0$.
\end{lemma}

\section{Main results}\label{sec: main}

In this section, we begin with Theorem \ref{Thm: tail asymptotics} below for the asymptotic expansion of the tail probability on the $A_{\vk x_S}$ and discuss the rate of decay 
in the probability of joint tail events. Based on the tail behaviour, we then characterize the multivariate regular variation on the sub-cones $\EE_d^{(i)}$'s.

Recall that ${\vk X} \sim F$ has heavy-tailed margins and the same elliptical copula as an elliptical random vector $\vk Z \stackrel{d}{=} R A{\bf U}$. All the parameters $(C,L,\beta, \gamma, C',L',\beta', \gamma')$ for Weibullian-type radius $R$ defined in Eq.\eqref{R} and for margins $Z_j$ in Eq.\eqref{Z_j} are used throughout the following paper.

\begin{theorem}\label{Thm: tail asymptotics}
Let $\vk X \in \RVEC(\alpha, b, \Sigma,R)$ with $R$ being Weibullian-type and let $\Sigma \in \R^{d\times d}$ be a positive-definite correlation matrix. For a non-empty set $S \subseteq \{1, \ldots, d\}$ with $|S| \geq 2$, and let 
$(\lambda_S, I_S, J_S, \vk e_S^*, \Upsilon_S,h_j^S, \beta_{\gamma,I_S,J_S})$ be defined w.r.t. $\Sigma_S$ as in Lemma \ref{Lemma: quadratic} and Proposition \ref{Proposition: type I}. Let $A_{\vk x_S}$ be defined in Eq.\eqref{def: tail set} and suppose further that
\begin{eqnarray*}\overline{F}_j(t)\sim c_j\overline{F}_1(t)=c_j\left(t^\alpha \ell(t)\right)^{-1},\, c_j>0, \quad \log (\ell(t))=o({\log (t)}^{1-1/\gamma}).\end{eqnarray*}
Then
\begin{eqnarray*}
\pk{\vk X \in tA_{\vk x_S}}
&=& (1+o(1)) \Upsilon_S \lambda_S^{\beta_{\gamma,I_S,J_S}} \fracl{\alpha}{L}^{\frac{\beta}{\gamma}+1-\frac{|I_S|}{2}-\frac{|S|}{2}}\\
&&\times (\log t)^{\frac{\beta}{\gamma}+1-\frac{|I_S|}{2}-\frac{|S|}{2}-\frac{\beta'}{\gamma}\lambda_S^{\gamma/2}} (t^\alpha \ell(t))^{-\lambda_S^{\gamma/2}} \\
&& \times (C' \alpha^{\beta'/\gamma} L^{-\beta'/\gamma})^{-\lambda_S^{\gamma/2}} \prod_{j \in I_S} (c_j^{-1}x_j^\alpha)^{- h_j^S\lambda_S^{\gamma/2-1}}.   
\end{eqnarray*}
\end{theorem}

Note that all type I elliptical copulas admit the asymptotic (tail) independence property \citep{Schmidt2002, FRAHM2003}, i.e., $\pk{X_j > tx_j, \forall j\in S}=o(\pk{X_1 > tx_1})$ as $t \to \infty$, for the tail equivalent risk $\vk X$. The dominant term $(t^\alpha \ell(t))^{-\lambda_S^{\gamma/2}}$ is regularly varying with index $-\alpha\lambda_S^{\gamma/2} \le -\alpha, \lambda_S\ge1, \gamma>0$, and $\lambda_S^{\gamma/2}$ indicates the rate of decay in the joint tail probability compared with the marginal risk. Moreover, the rate only depends on the interested set $S$ and the elliptical copula determined by $\Sigma$ and $R$, with $\lambda_S=\lambda_S(\Sigma_S)$ and $\pk {R>u} \sim Cu^\beta \exp(-Lu^\gamma)$.

\begin{remark}\label{Re: gamma}
\kai{a) Many distributions, e.g., Pareto, Burr, L\'{e}vy, Student's $t$ and Fr\'{e}chet distributions, have regular varying tails satisfying $\overline F(t) = x^{-\alpha} \ell(t)$ such that 
\begin{eqnarray*}\log \ell(t)=o({\log (t)}^{1-1/\gamma}), \gamma>1.\end{eqnarray*}
b) Theorem \ref{Thm: tail asymptotics} holds also for $\gamma=1$ if  $\ell(t) \sim c$ for some constant $c>0$. In this case, $\log \mathcal{L}(t)= (\log \ell(t))/L + (\beta'\log\log t)/L$ and the two terms are absorbed into $\vk z_S$ and $u(t)$ in Eq.\eqref{Eq: tail expansion}, respectively. Therefore, $u(t) \sim \log t$ as $\log \log t = o(\log t)$. See the Laplace copula for instance in Example \ref{Ex: Laplace}.}
\end{remark}


\begin{remark} \cl{Our results extend the Gaussian dependence structure studied by \cite{Das2023heavy} to be a general elliptical copula with wide applications (cf. Section \ref{sec: app}). Indeed, a straightforward application of our result into Gaussian copula with} $ L=1/2, \ \gamma=2, \ \beta' =-1, \ C'= 1/\sqrt{2\pi}$, and identically distributed marginal risks (e.g., all $c_j$'s equal 1), indicates
that
\begin{eqnarray}\label{DasThm3.1}
\pk{\vk X \in tA_{\vk x_S}} &\sim& \Upsilon_S (2\alpha)^{-\frac{|I_S|}{2}} (\log t)^{-\frac{|I_S|}{2}+\frac{\lambda_S}{2}} (t^\alpha\ell(t))^{-\lambda_S}(2\sqrt{\pi \alpha})^{\lambda_S} \prod_{j \in S} x_j^{-\alpha h_j^S}   \notag\\
&\sim& \Upsilon_S (2\pi )^{\frac{\lambda_S}{2}} (2\alpha\log t)^{\frac{\lambda_S-|I_S|}{2}} (t^\alpha\ell(t))^{-\lambda_S} \prod_{j \in S} x_j^{-\alpha h_j^S}.
\end{eqnarray}
We see that Eq.\eqref{DasThm3.1} is exactly Theorem 3.1 in \cite{Das2023heavy}.
\end{remark}

Below, we will consider the multivariate regular variation on the sub-cones $\EE_d^{(i)}$ and present our second main theorem.
\begin{theorem}\label{Thm: MRV}
Let $\vk X \in \RVEC(\alpha, b, \Sigma,R)$ where $R$ is Weibullian-type and $\Sigma \in \R^{d\times d}$ is positive-definite correlation matrix. Under the conditions of Theorem \ref{Thm: tail asymptotics},\\
$(a)$ For subcone $\EE_d^{(1)}$, $\vk X \in \MRV(\alpha,b,\nu_1,\EE_d^{(1)})$ with
\begin{eqnarray*}
\nu_1\left([\vk 0, \vk x]^c\right)= \sum\limits_{j=1}^d x_j^{-\alpha}, \quad \forall \vk x \in \R_+^d.
\end{eqnarray*}
$(b)$ For subcones $\EE_d^{(i)}$, $2\le i \le d$, define
\begin{eqnarray*}
\mathcal{S}_i &:=& \left\{ S \subset \mathbb{I}: |S|\ge i, \vk 1_{I_S}^{\top} \Sigma_{I_S}^{-1}\vk 1_{I_S}= \min\limits_{\widetilde{S} \subset \mathbb{I}, |\widetilde{S}| \ge i } \vk 1_{I_{\widetilde{S}}}^{\top} \Sigma^{-1}_{I_{\widetilde{S}}} \vk 1_{I_{\widetilde{S}}} \right\},\\
I_i &:=& \arg\min\limits_{S \in \mathcal{S}_i} |I_S|.
\end{eqnarray*}
Then $\vk X \in \MRV(\alpha_i,b_i,\nu_i,\EE_d^{(i)})$ with
\begin{eqnarray*}
\lambda_i &=& \lambda(\Sigma_{I_i}) = \min\limits_{S\subset\mathbb{I},|S|\ge i}\min\limits_{\vk x_S\ge \vk 1_S}\vk x_S^\top \Sigma_S^{-1} \vk x_S, \\
\alpha_i &=& \alpha\lambda_i^{\gamma/2} ,\\
b_i^{\leftarrow}(t) &=&  C^{'\lambda_i^{\gamma/2}} \fracl{\alpha\log t}{L}^{\frac{|I_i|-1}{2}-\frac{\beta'}{\gamma}(1-\lambda_i^{\gamma/2})}  (b^{\leftarrow}(t))^{\lambda_i^{\gamma/2}},
\end{eqnarray*}
and for the set $A_{\vk x_S}=\left\{\vk {y} \in \R_{+}^d: y_j>x_j, \forall j \in S\right\}$ with $\vk x_S>\mathbf0$ for $S \subseteq \mathbb{I}$ and $|S| \geq i$, we have
\begin{eqnarray*}
\nu_i(A_{\vk x_S}) &=& 
\begin{cases}
 \Upsilon_S \lambda_S^{\beta_{\gamma,I_S,J_S}} \prod_{j \in I_S} (c_j^{-1}x_j^\alpha)^{- h_j^S\lambda_S^{\gamma/2-1}}, & \ \text{if} \ S \in \mathcal{S}_i \ \text{and} \ |I_S|=|I_i|, \\
0, &\ \text{otherwise}. 
\end{cases}
\end{eqnarray*}
\end{theorem}

\begin{remark}
    \kai{The rates of decay in the sub-cones $\EE_d^{(i)}$ are shown with $\alpha_i=\alpha \lambda_i^{\gamma/2}, i=1,\ldots,d$. The result with $\gamma=2$ is reduced to the Gaussian-like case in \citet[Theorem 2]{Das2023heavy}. Moreover, for the case with $\gamma=1$, Theorem \ref{Thm: MRV} holds with the condition for $\ell(t)\sim c$ discussed in Remark \ref{Re: gamma}.}
\end{remark}

\section{Example}\label{sec: Ex}
A typical class of elliptical copulas is the so-called multivariate power exponential copula. The original distribution, also named the multivariate generalized Gaussian distribution and studied by \cite{KANO1994PED,Gomez1998PED}, are widely applied in biostatistics, bioinformatics, and engineering \citep{Dang2015}. 
A random vector $\vk{Z}$ follows a $d$-dimensional power exponential distribution, denoted by $\vk Z\sim PE_d(\boldsymbol{\mu}, {\Sigma}, \kappa)$, if its density is
\begin{eqnarray*}
f(\mathbf{z}, \boldsymbol{\mu}, {\Sigma}, \kappa)=c_d|{\Sigma}|^{-\frac{1}{2}} \exp \left\{-\frac{1}{2}\left[(\mathbf{z}-\boldsymbol{\mu})^{\prime} {\Sigma}^{-1}(\mathbf{z}-\boldsymbol{\mu})\right]^\kappa\right\},
\end{eqnarray*}
where $c_d={d \Gamma(d/2)}/\left(\pi^{d/2} \Gamma(1+{d}/({2 \kappa})) 2^{1+{d}/({2 \kappa})}\right)$. Moreover, the $\vk Z$ is an elliptical risk with ${\vk Z} \stackrel{d}{=} \vk \mu + R A{\bf U}$. Here the radius $R>0$, being independent of uniform distributed $\bf U$ in $\mathbb S^{d-1}$, has the following density function 
\begin{eqnarray}\label{Eq: representation}
    h_R(r)=\frac{d}{\Gamma\left(1+\frac{d}{2 \kappa}\right) 2^{\frac{d}{2 \kappa}}} r^{d-1} \exp \left\{-\frac{1}{2} r^{2 \kappa}\right\},\quad r>0.
\end{eqnarray}
Power exponential distributions consist of many well-known risks, including the Laplace distribution ($\kappa=1/2$) and Gaussian distribution ($\kappa=1$). In the following context, we first verify the radius $R$ is Weibullian-type by Mill's ratio, then discuss the solution to the quadratic programming problem with a correlation matrix $\Sigma \in \R^{3\times 3}$ and equi-correlation matrix $\Sigma_\rho \ (\rho_{jk}=\rho \mbox{ for } j \ne k \in \{1,2,3\}$). Finally, we illustrate the examples of the power exponential copula, Laplace copula and Gaussian copula.

\begin{lemma}[Mills’ ratio inequality]\label{Lemma: MR for R} Let $h(r)$ be the density of radius $R$ in Eq.\eqref{Eq: representation}. For $r>0$, the inequalities 
\begin{eqnarray*} 
\frac{r}{\kappa(r^{2\kappa}+2)} h(r)\le \int_{r}^{\infty} h(u) du \le \frac{1}{\kappa r^{2\kappa-1}} h(r)
\end{eqnarray*}
hold and for $r\to \infty$,
\begin{eqnarray*}  
\int_{r}^{\infty} h(u)du \approx \frac{1}{\kappa r^{2\kappa-1}}h(r).
\end{eqnarray*}
\end{lemma}

\begin{remark}
  For a multivariate power exponential random vector $\vk Z=RA\vk U$ has a density in the form of Eq.\eqref{Eq: representation}. For large $u>0$, by Theorem \ref{Lemma: MR for R},
\begin{eqnarray*}\pk{R>u} \sim \frac{1}{\Gamma\left({d}/{\gamma}\right) 2^{\frac{d}{\gamma}-1}} u^{d-\gamma} \exp \left\{-\frac{1}{2} u^{\gamma}\right\}, \quad \gamma=2\kappa>0,\end{eqnarray*}
and 
\begin{eqnarray*}
\pk{Z_j>u} \sim \frac{\Gamma(d/2)}{\Gamma(1/2)\Gamma(d/\gamma)2^{d/\gamma}} \fracl{\gamma L}{2}^{-(d-1)/2} u^{d-\gamma(d+1)/2}\exp \left\{-\frac{1}{2} u^{\gamma}\right\},
\end{eqnarray*}
indicating that the radius in the power exponential distribution is  Weibullian-type.  
\end{remark}

For a positive definite correlation matrix
\begin{eqnarray*}
\Sigma = 
\begin{pmatrix}
1 & \rho_{12} & \rho_{13}\\
\rho_{12} & 1 & \rho_{23}\\
\rho_{13} & \rho_{23} & 1
\end{pmatrix},
\end{eqnarray*}
where $\rho_{jk} \in (-1,1), \forall j \ne k$, we have $|\Sigma|=1-\rho_{12}^2-\rho_{13}^2-\rho_{23}^2+2\rho_{12}\rho_{13}\rho_{23}$ and 
\begin{eqnarray*}
\Sigma^{-1} = \frac{1}{|\Sigma|}
\begin{pmatrix}
1-\rho_{23}^2 & \rho_{13}\rho_{23}-\rho_{12} & \rho_{12}\rho_{23}-\rho_{13}\\
\rho_{13}\rho_{23}-\rho_{12} & 1-\rho_{13}^2 & \rho_{12}\rho_{13}-\rho_{23}\\
\rho_{12}\rho_{23}-\rho_{13} & \rho_{12}\rho_{13}-\rho_{23} & 1-\rho_{12}^2
\end{pmatrix}.
\end{eqnarray*}
Suppose $|\Sigma|>0$ and $\vk h:=\Sigma^{-1}\vk 1 > \vk 0$. For symmetric $\Sigma^{-1}$, $h_i:= \vk e_i^\top\Sigma^{-1}\vk 1 = \vk 1^\top\Sigma^{-1}\vk e_i, \ i =1,2,3$, which is defined in Lemma \ref{Lemma: type I}, then applying Theorem \ref{Thm: MRV},
\begin{eqnarray*}
\Sigma_{I_2}^{-1} &=& \frac{1}{1-\rho_{*}^2}
\begin{pmatrix}
1 & -\rho_{*} \\
-\rho_{*}  & 1 
\end{pmatrix}, \quad \rho_* = \max\limits_{1\le j \ne k \le 3} \rho_{jk},
\end{eqnarray*}
and 
\begin{eqnarray*}
\lambda_1 &=& 1, \\
\lambda_2&=&\lambda(\Sigma_{I_2})=\vk 1_{I_2}^\top\Sigma_{I_2}^{-1}\vk 1_{I_2}= \frac{2}{1+\rho_{*}},\\
\lambda_3&=&\lambda(\Sigma)=\vk 1^\top\Sigma^{-1}\vk 1.
\end{eqnarray*}

For an equi-correlation matrix $\Sigma_\rho$ and $|\Sigma_\rho|=(1-\rho)^2(1+2\rho)$, let $\rho \in (-1/2,1)$ be a constant keeping the matrix positive definite, then 
\begin{eqnarray*}
\Sigma_\rho^{-1} &=& \frac{1}{(1-\rho)^2(1+2\rho)}
\begin{pmatrix}
1-\rho^2 & \rho^2-\rho & \rho^2-\rho \\
\rho^2-\rho  & 1-\rho^2 & \rho^2-\rho \\
\rho^2-\rho  & \rho^2-\rho  & 1-\rho^2
\end{pmatrix},\\
\vk h &:=& \Sigma_\rho^{-1}\vk 1= 
\begin{pmatrix}
\frac{1}{1+2\rho} \\
\frac{1}{1+2\rho} \\
\frac{1}{1+2\rho}
\end{pmatrix}> \vk 0.
\end{eqnarray*}
 
 By Theorem \ref{Thm: MRV}, one choice is $I_2=\{1,2\} \in \mathcal{S}_2$ with 
\begin{eqnarray*}
\Sigma_{I_2}^{-1} &=& \frac{1}{1-\rho^2}
\begin{pmatrix}
1 & -\rho \\
-\rho  & 1 
\end{pmatrix}
\end{eqnarray*}
and 
\begin{eqnarray*}
\lambda_1 &=& 1, \\
\lambda_2&=&\lambda(\Sigma_{I_2})=\vk 1_{I_2}^\top\Sigma_{I_2}^{-1}\vk 1_{I_2}= \frac{2}{1+\rho},\\
\lambda_3&=&\lambda(\Sigma)=\vk 1^\top\Sigma^{-1}\vk 1=\frac{3}{1+2\rho}.
\end{eqnarray*}

\begin{example}[Power exponential copula]\label{Ex: PE}
Suppose that $\vk X= (X_1, X_2, X_3)\in \RVEC(\alpha, b, \Sigma_\rho, R)$ has the regularly varying marginal tails and the power exponential copula, where $\Sigma_\rho$ is positive definite with $-1/2<\rho<1$. Also assume that $\overline{F}_1(t)=\left(t^\alpha \ell(t)\right)^{-1}$ satisfying $\log (\ell(t))=o({\log (t)}^{1-1/\gamma}), \gamma>1$.
Thus, \cl{it follows by Theorems \ref{Thm: tail asymptotics} and \ref{Thm: MRV} that, for $S \subset \{1,2,3\}$,}
\begin{eqnarray*}
\pk{\vk X \in tA_{\vk x_S}}&\sim& \Upsilon_S \lambda_S^{\beta_{\gamma,I_S,J_S}} (2\alpha\log t)^{\frac{3}{\gamma}-\frac{|I_S|}{2}-\frac{|S|}{2}-\left(2-\frac{3}{\gamma}\right)\lambda_S^{\gamma/2}} \\
&& \times (t^\alpha \ell(t))^{-\lambda_S^{\gamma/2}} (2^{\frac{3-\gamma}{\gamma}}\Gamma(3/{\gamma}) \gamma )^{\lambda_S^{\gamma/2}} \prod_{j \in I_S} (c_j^{-1}x_j^\alpha)^{- h_j^S\lambda_S^{\gamma/2-1}}   
\end{eqnarray*}
and $\vk X \in \MRV(\alpha_i,b_i,\nu_i,\EE_3^{(i)}), i=1,2,3$, with
\begin{eqnarray*}
&&\lambda_1 = 1, \ \lambda_2=\vk 1_{I_2}^\top\Sigma_{I_2}^{-1}\vk 1_{I_2}= \frac{2}{1+\rho},\ \lambda_3=\vk 1^\top\Sigma^{-1}\vk 1= \frac{3}{1+2\rho},\\
&&\alpha_1 = \alpha, \ \alpha_2 = \alpha\lambda_2^{\gamma/2}=\alpha\fracl{2}{1+\rho}^{\gamma/2}, \ \alpha_3 = \alpha\lambda_3^{\gamma/2}=\alpha\fracl{3}{1+2\rho}^{\gamma/2},\\
&&b_1^{\leftarrow}(t) = b^{\leftarrow}(t), \\
&&b_2^{\leftarrow}(t) =  (2^{\frac{3-\gamma}{\gamma}}\Gamma(3/{\gamma}) \gamma)^{-\lambda_2^{\gamma/2}} (2\alpha\log t)^{\left(\frac{3}{\gamma}-2\right)\lambda_2^{\gamma/2}+\frac{5}{2}-\frac{3}{\gamma}}  (b^{\leftarrow}(t))^{\lambda_2^{\gamma/2}}, \\ 
&&b_3^{\leftarrow}(t) =  (2^{\frac{3-\gamma}{\gamma}}\Gamma(3/{\gamma}) \gamma)^{-\lambda_3^{\gamma/2}} (2\alpha\log t)^{\left(\frac{3}{\gamma}-2\right)\lambda_3^{\gamma/2}+3-\frac{3}{\gamma}}  (b^{\leftarrow}(t))^{\lambda_3^{\gamma/2}}
\end{eqnarray*}
and
\begin{eqnarray*}
&&\nu_1\left([\vk 0, \vk x]^c\right) = \sum\limits_{j=1}^3 x_j^{-\alpha},\\
&&\nu_2(A_{\vk x_S}) =
\begin{cases}
 \Upsilon_S \lambda_S^{\beta_{\gamma,I_S,J_S}} \prod_{j \in I_S} (c_j^{-1}x_j^\alpha)^{- h_j^S\lambda_S^{\gamma/2-1}}, & \ \text{if} \ S \in \mathcal{S}_2 \ \text{and} \ |I_S|=2, \\
0, &\ \text{otherwise},
\end{cases}\\
&&\nu_3(A_{\vk x_S}) =
\begin{cases}
 \Upsilon_S \lambda_S^{\beta_{\gamma,I_S,J_S}} \prod_{j \in I_S} (c_j^{-1}x_j^\alpha)^{- h_j^S\lambda_S^{\gamma/2-1}}, & \ \mbox{if} \ S \in \mathcal{S}_3 \ \mbox{and} \ |I_S|=3, \\
0, &\ \mbox{otherwise}.
\end{cases}
\end{eqnarray*}

\end{example}

\begin{example}[Laplace copula]\label{Ex: Laplace}
Suppose $\vk X= (X_1, X_2, X_3)\in \RVLC(\alpha, b, \Sigma_\rho,R)$ with regularly varying marginal tails and the Laplace copula $(\gamma=1$ in the power exponential copula$)$, where $\Sigma_\rho$ is positive definite.  Assume that for $\overline{F}_1(t)=\left(t^\alpha \ell(t)\right)^{-1}$, $\ell(t) \sim c$  as $t \rightarrow \infty$, where $c$ is a constant. Then all the properties hold with $\gamma=1$ in Example \ref{Ex: PE}.

\end{example}

\section{Simulation}\label{sec: Simulation}

In this section, we consider the random vector $\vk X= (X_1, X_2, X_3)\in \RVEC(\alpha, b, \Sigma_\rho,R)$ with identical Pareto distributed margins $(\alpha=2)$ and correlation matrix $\Sigma_\rho$ $(\rho=0.6,0.8)$ with power exponential copulas ($\gamma=1,2,3$), where $\gamma=1,2$ refers to the Laplace copula and  Gaussian copula, respectively. We verify the theoretical tail indices ($\alpha_i$) using the Hill estimators \citep{Hill1975,Beirlant2004}  on the following sets with simulated data ($n=20,000$, 
notation: $y_{(1)}\ge y_{(2)} \ge y_{(3)}$, the upper ranking of $y_1,y_2, y_3>0$):
\begin{enumerate}
    \item In $\EE_3^{(1)}$ : $A_j=\left\{\boldsymbol{y} \in \mathbb{R}_{+}^3: y_j>1\right\}$ and $A_{(1)}=\left\{\vk {y} \in \R_{+}^3: y_{(1)}>1\right\}$
with $X_j$ for $j=1,2,3$ and $X_{(1)}$.
\item 
In $\EE_3^{(2)}$: $A_{j k}=\left\{\vk {y} \in \R_{+}^3: y_j>1, y_k>1\right\}$ and $A_{(2)}=\left\{\vk{y} \in \R_{+}^3: y_{(2)}>1\right\}$ \\with $\min(X_j,X_k)$ for $1\le j \ne k \le 3$ and $X_{(2)}$.
\item
In $\EE_3^{(3)}$: $A_{(3)}=\left\{\vk {y} \in \R_{+}^3: y_{(3)}>1\right\}$ and $X_{(3)}$.
\end{enumerate}
Results of Case (1), (2), and (3) are shown in subfigures (a,c), (b,c), and (c) of Figures \ref{fig1} and \ref{fig2}, respectively. Figures \ref{fig1} and \ref{fig2} show the Hill estimates of tail indices for the Pareto risks $(\alpha=2)$ under the Laplace copula with $\rho=0.6,0.8$. Figures \ref{fig1:a} and \ref{fig2:a} illustrate the stable trends near the true value $\alpha=2$ for the marginal risks. The same trends also appear in the plots for the max risk $X_{(1)}$, see the blue lines in Figures \ref{fig1:c} and \ref{fig2:c},
supporting the theoretical result of regular variation with index $\alpha_1=2$ in $\EE_3^{(1)}$. As for $\EE_3^{(2)}$, suggested by Figures \ref{fig1:b} and  \ref{fig2:b} and the red lines in Figures \ref{fig1:c} and \ref{fig2:c}, all the tail indices are between 2 and 2.5, while the indices for $\rho=0.6$ are a bit larger than the case of $\rho=0.8$. Such results are also supported by $\alpha_2 = \sqrt{2}\alpha/(1+\rho)^{1/2}=2.2, \rho=0.6$, and $\alpha_2 =2.1, \rho=0.8$ given in Example \ref{Ex: Laplace}. The Hill estimates for the tail index of the minimum risk $X_{(3)}$ are shown with the green lines in Figures \ref{fig1:c} and \ref{fig2:c}. Both estimates are the highest in their cases, reaching approximately 2.5 for $\rho=0.6$
and close to $2.2$ for $\rho=0.8$. The theoretical values are $\alpha_3 = \sqrt{3}\alpha/(1+2\rho)^{1/2}=2.4,\rho=0.6$ and $\alpha_3 =2.2,\rho=0.8$ from  Example \ref{Ex: Laplace}, supported the simulation result.

In the cases of Pareto risks $(\alpha=2)$ with Gaussian copula and power exponential copula ($\gamma=3$), for simplicity, we only present the tail indices for $X_{(1)},X_{(2)},X_{(3)}$ in Figures \ref{fig3} and \ref{fig4}, representing Cases (1), (2), (3), respectively. The theoretical indices for the Gaussian case are $\alpha_1=2,\alpha_2=2\alpha/(1+\rho)=2.5,\alpha_3=3\alpha/(1+2\rho)=2.7, \rho=0.6$ and $\alpha_1=2,\alpha_2=2.2,\alpha_3=2.3, \rho=0.8$, which are justified by the Hill plots in Figures \ref{fig3:a} and \ref{fig3:b}. The Hill estimates for the tail index of $X_{(1)}, X_{(2)}$ are quite close to the theoretical values. Although the estimates for $X_{(3)}$ are slightly higher, the 95\% confidence intervals (not shown in the plot) still cover the theoretical results.
Similar results are observed with the power exponential copula ($\gamma=3$), where the plots in Figure \ref{fig4} align with the theoretical indices  $\alpha_1=2,\alpha_2= 2^{3/2}\alpha/(1+\rho)^{3/2}=2.8,\alpha_3=3^{3/2}\alpha/(1+2\rho)^{3/2}=3.2,\rho=0.6$ and $\alpha_1=2,\alpha_2=2.3,\alpha_3=2.5,\rho=0.8$, calculated in Example \ref{Ex: PE}.

In summary, the simulation results illustrate well the theoretical results presented for the heavy-tail risks with Laplace, Gaussian, and power exponential copulas discussed in Section \ref{sec: Ex}.

\begin{figure} [ht]
	\centering
	\subfloat[\label{fig1:a}]{
		\includegraphics[scale=0.26]{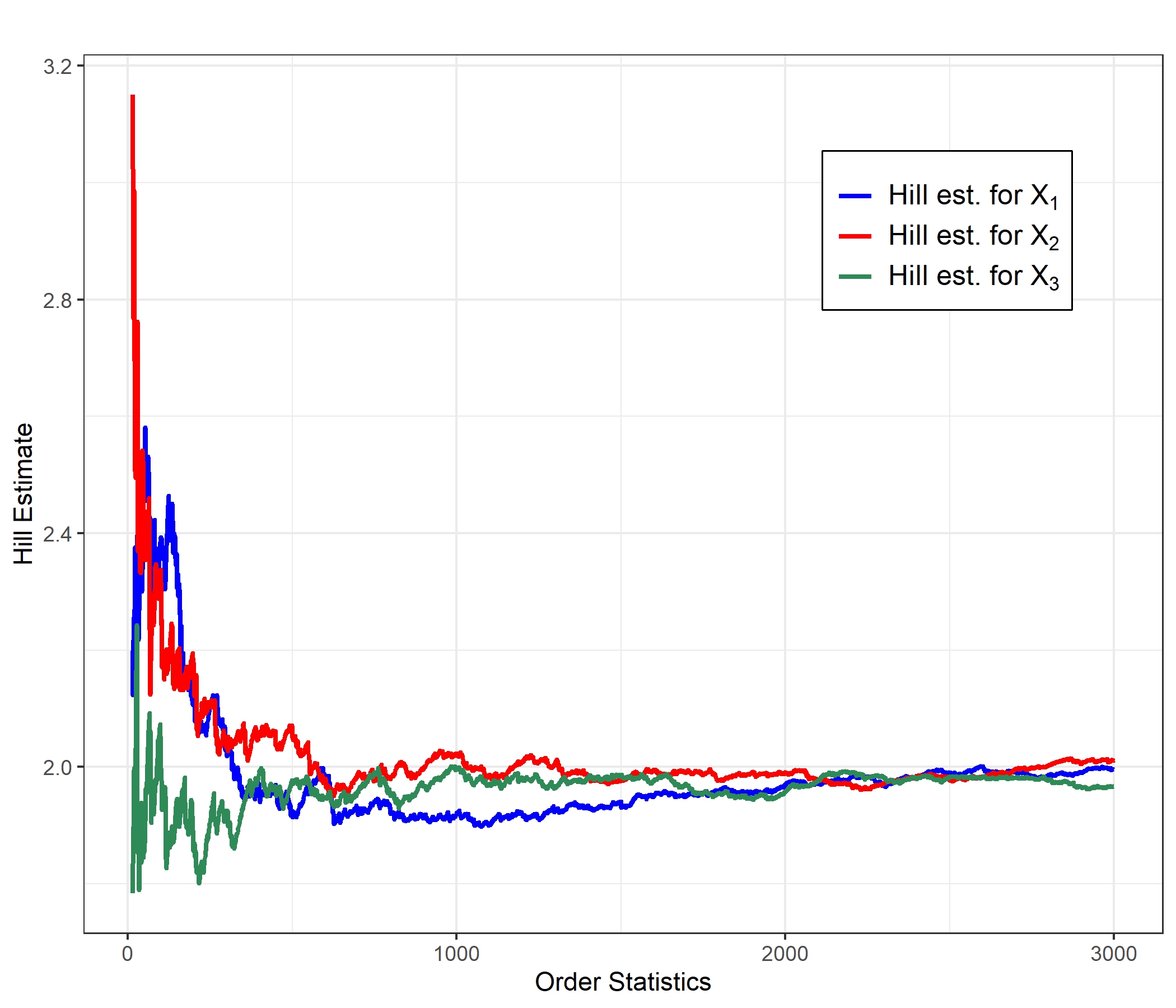}}
	\subfloat[\label{fig1:b}]{
		\includegraphics[scale=0.26]{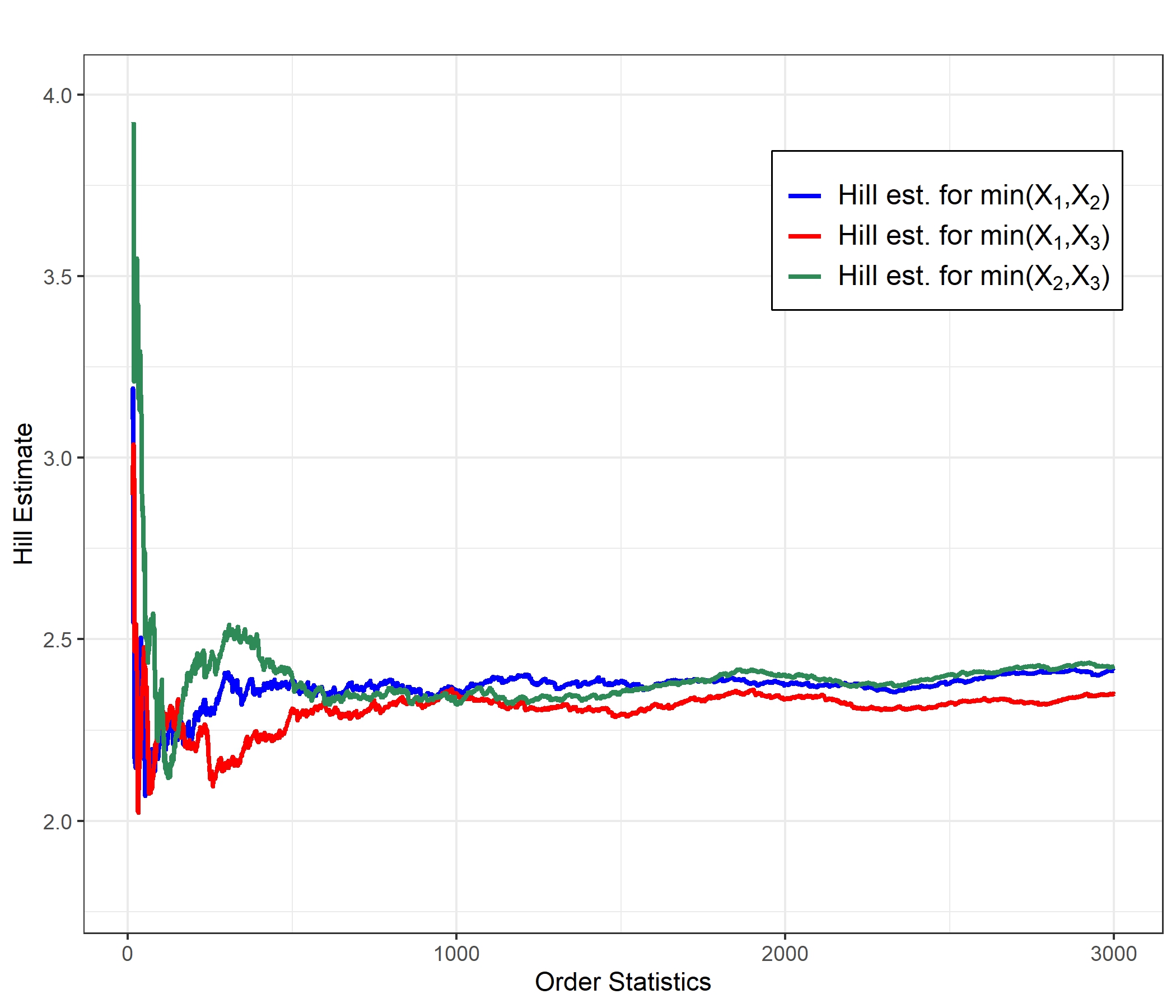}}
  	\subfloat[\label{fig1:c}]{
		\includegraphics[scale=0.26]{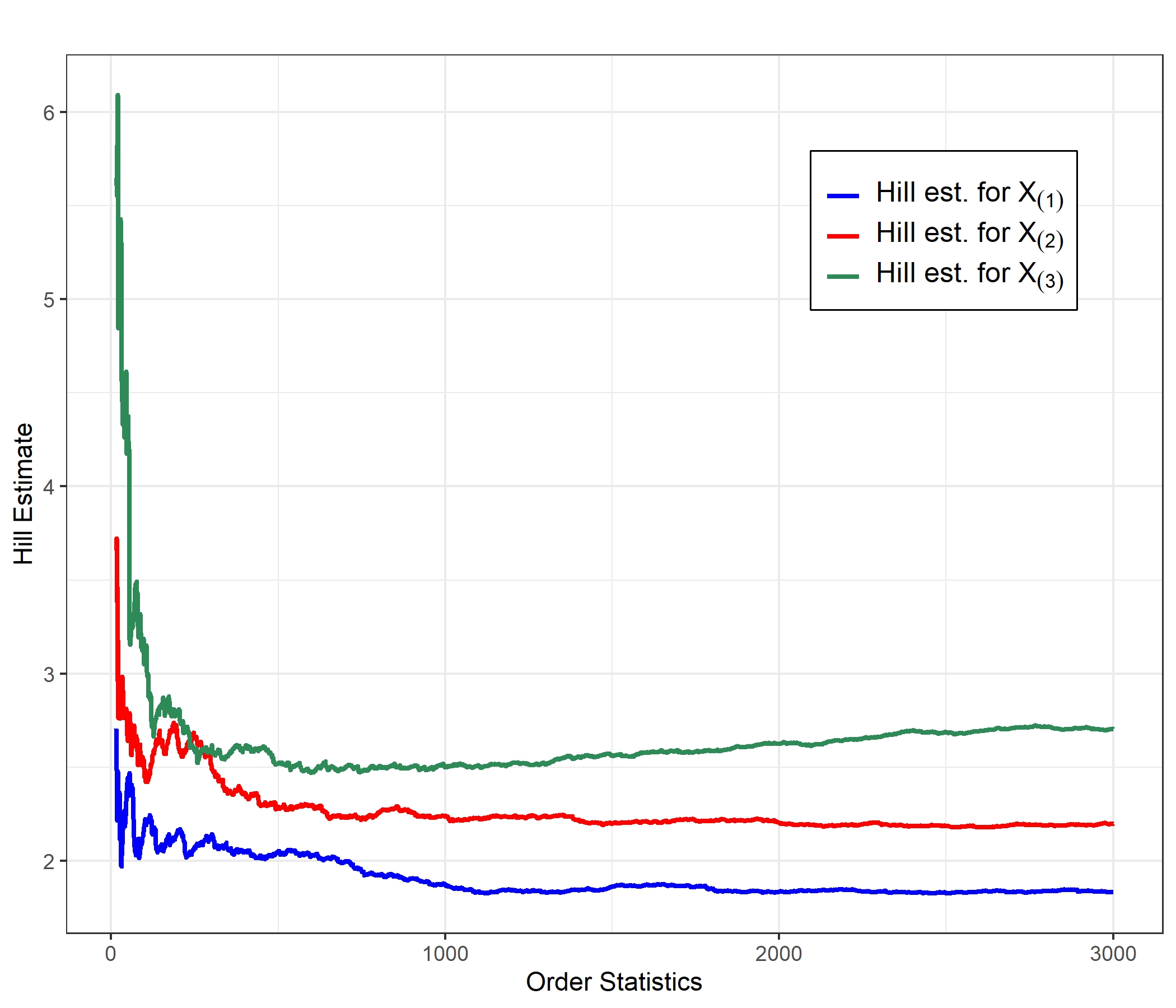}}
	\caption{Hill's plots for the Pareto risks ($\alpha=2$) with Laplace copula ($\rho=0.6$). The tail indices are $\alpha_1=2$ for marginal risks $X_j$ in (a), $\alpha_2=2.2$ for minimum of two risks $\min(X_j, X_k)$ in (b), $\alpha_1=2,\alpha_2=2.2,\alpha_3=2.4$ for the order statistics in (c).
 }\label{fig1}
\end{figure}

\begin{figure} [ht]
	\centering
	\subfloat[\label{fig2:a}]{
		\includegraphics[scale=0.26]{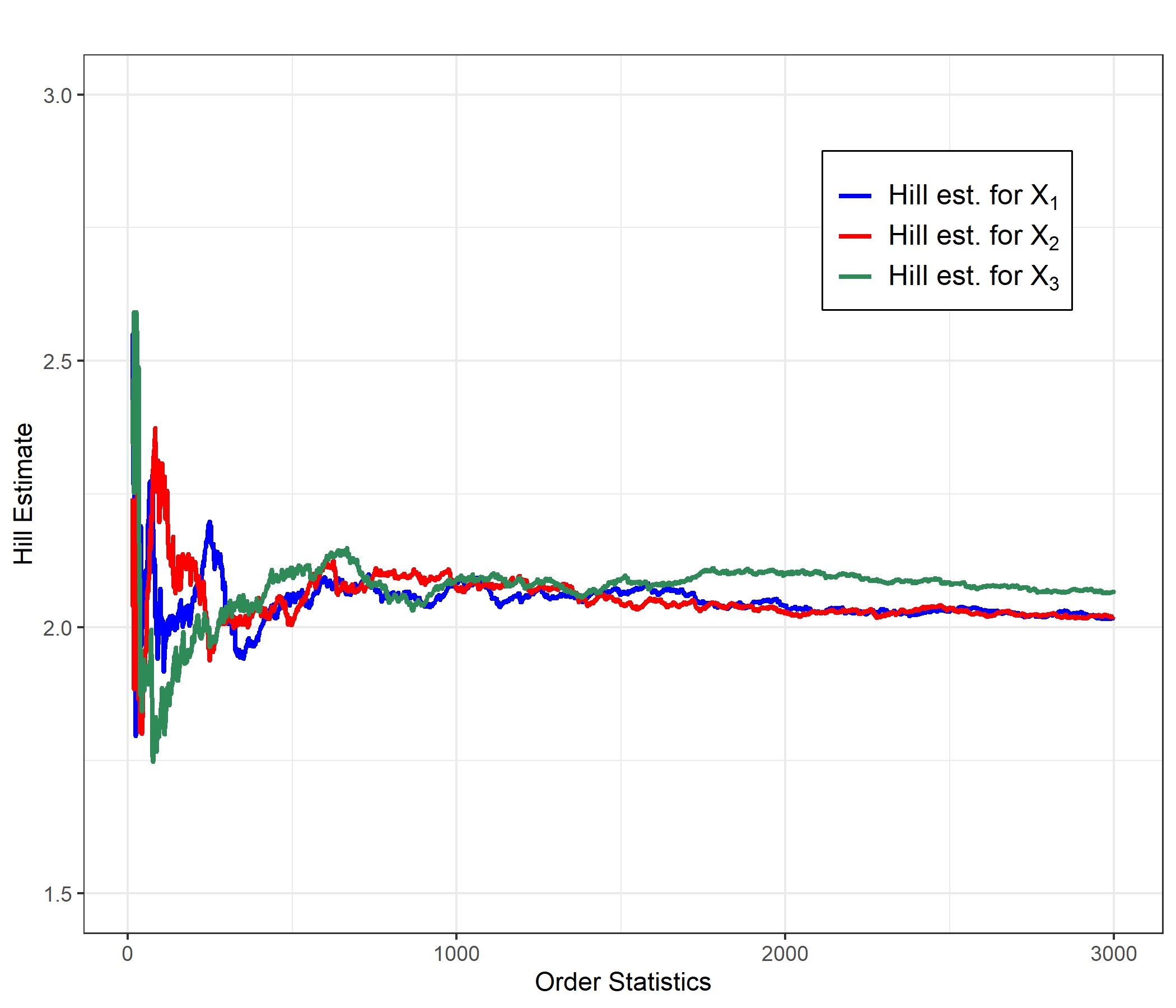}}
	\subfloat[\label{fig2:b}]{
		\includegraphics[scale=0.26]{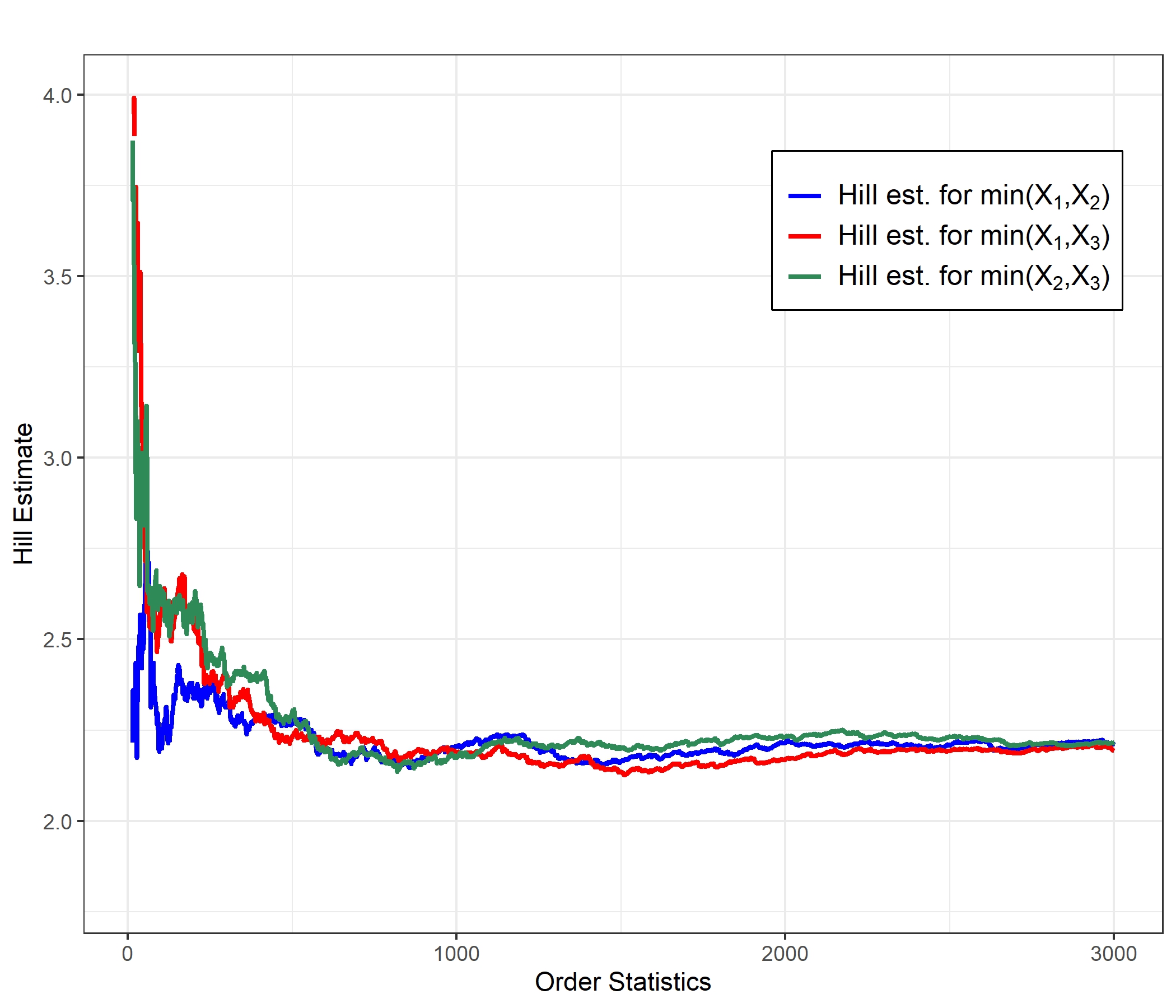}}
  	\subfloat[\label{fig2:c}]{
		\includegraphics[scale=0.26]{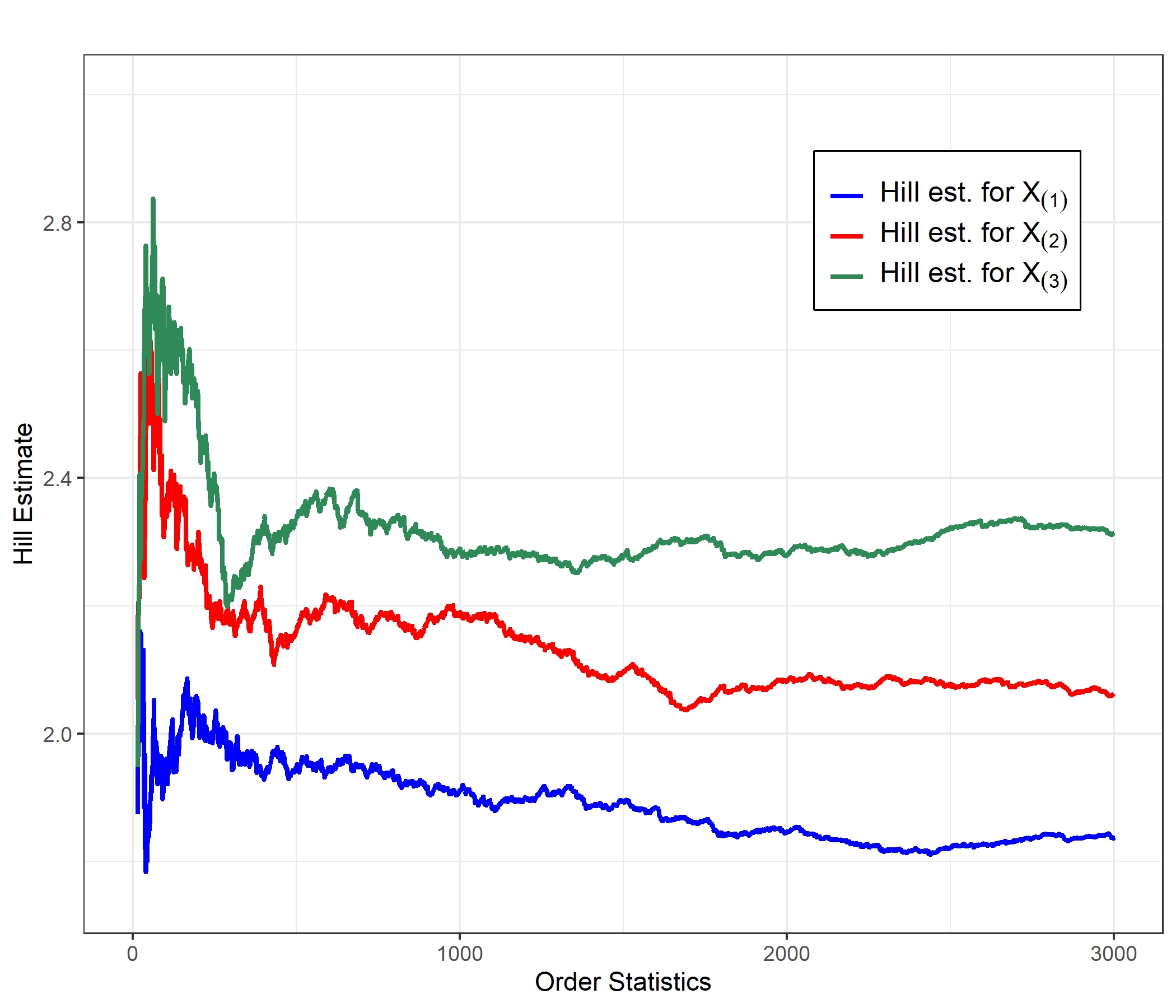}}
	\caption{Hill's plots for the Pareto risks ($\alpha=2$) with Laplace copula ($\rho=0.8$). The tail indices are $\alpha_1=2$ for marginal risks $X_j$ in (a), $\alpha_2=2.1$ for minimum of two risks $\min(X_j, X_k)$ in (b), $\alpha_1=2,\alpha_2=2.1,\alpha_3=2.2$ for the order statistics in (c).
 }\label{fig2}
\end{figure}

\begin{figure} [ht]
	\centering
	\subfloat[\label{fig3:a}]{
		\includegraphics[scale=0.35]{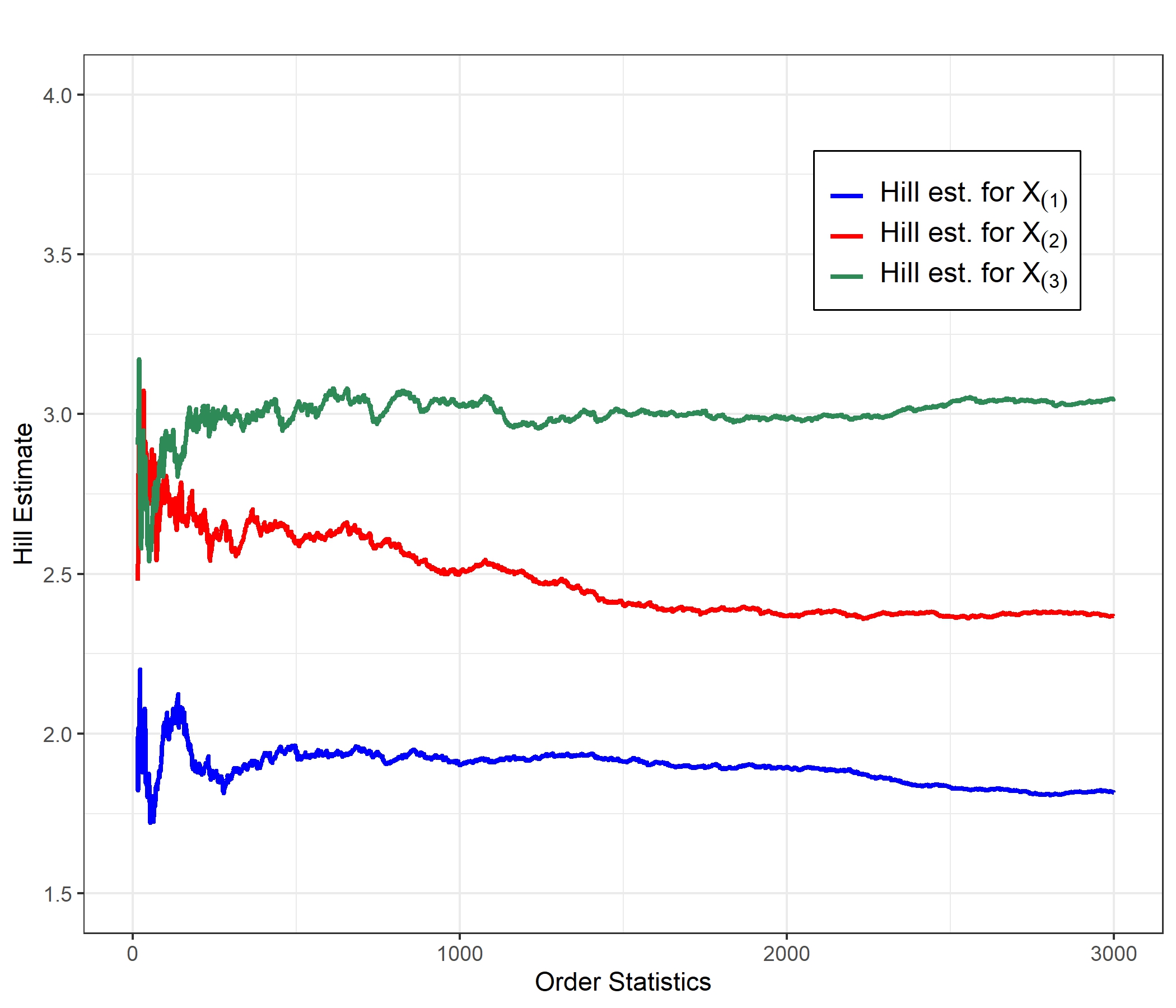}}
	\subfloat[\label{fig3:b}]{
		\includegraphics[scale=0.35]{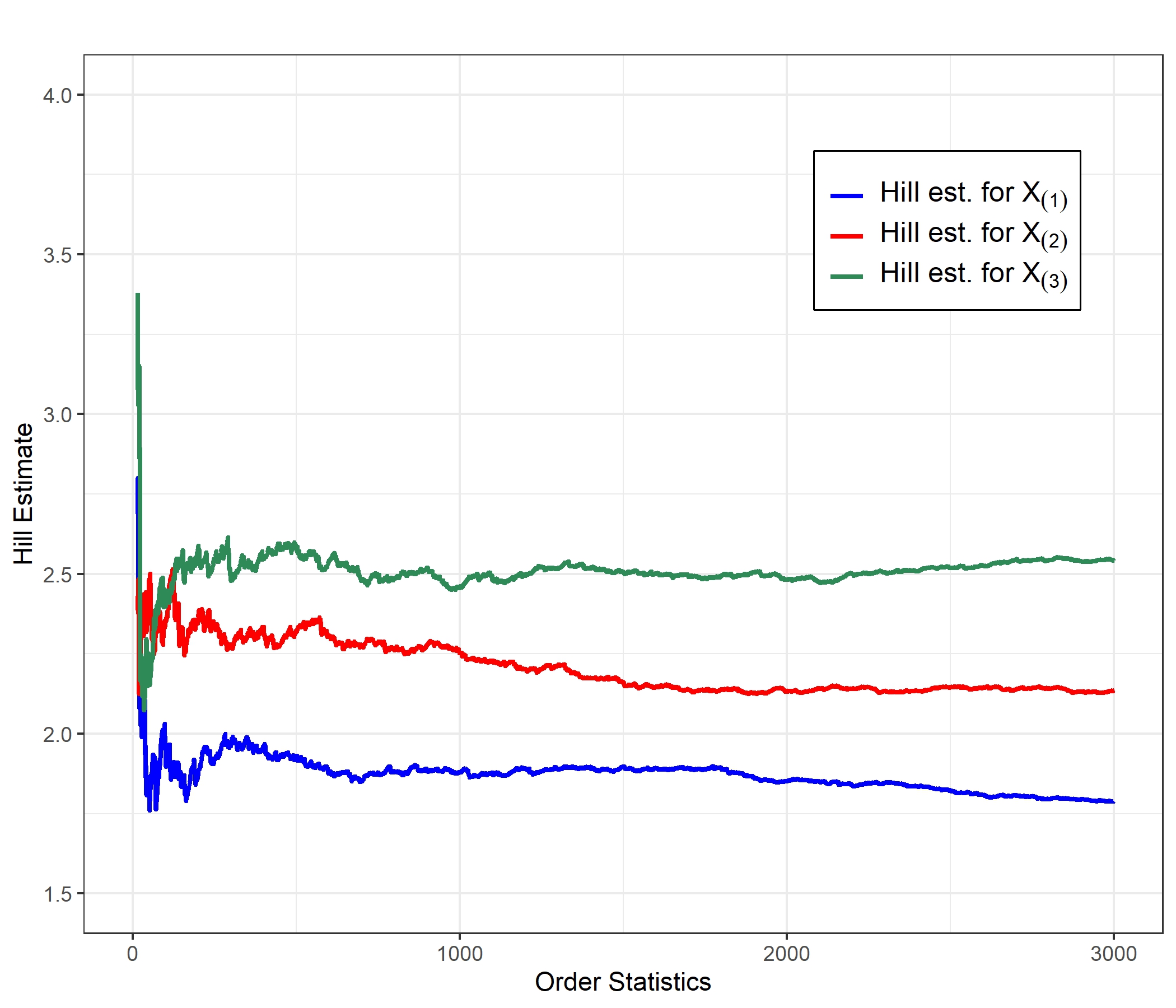}}
	\caption{Hill's plots for the Pareto risks ($\alpha=2$) with Gaussian copula, where $\rho=0.6$ in (a) and $\rho=0.8$ in (b). 
 The tail indices are $\alpha_1=2,\alpha_2=2.5,\alpha_3=2.7$ for the order statistics in (a) and $\alpha_1=2,\alpha_2=2.2,\alpha_3=2.3$ in (b).
 }\label{fig3}
\end{figure}

\begin{figure} [ht]
	\centering
	\subfloat[\label{fig4:a}]{
		\includegraphics[scale=0.35]{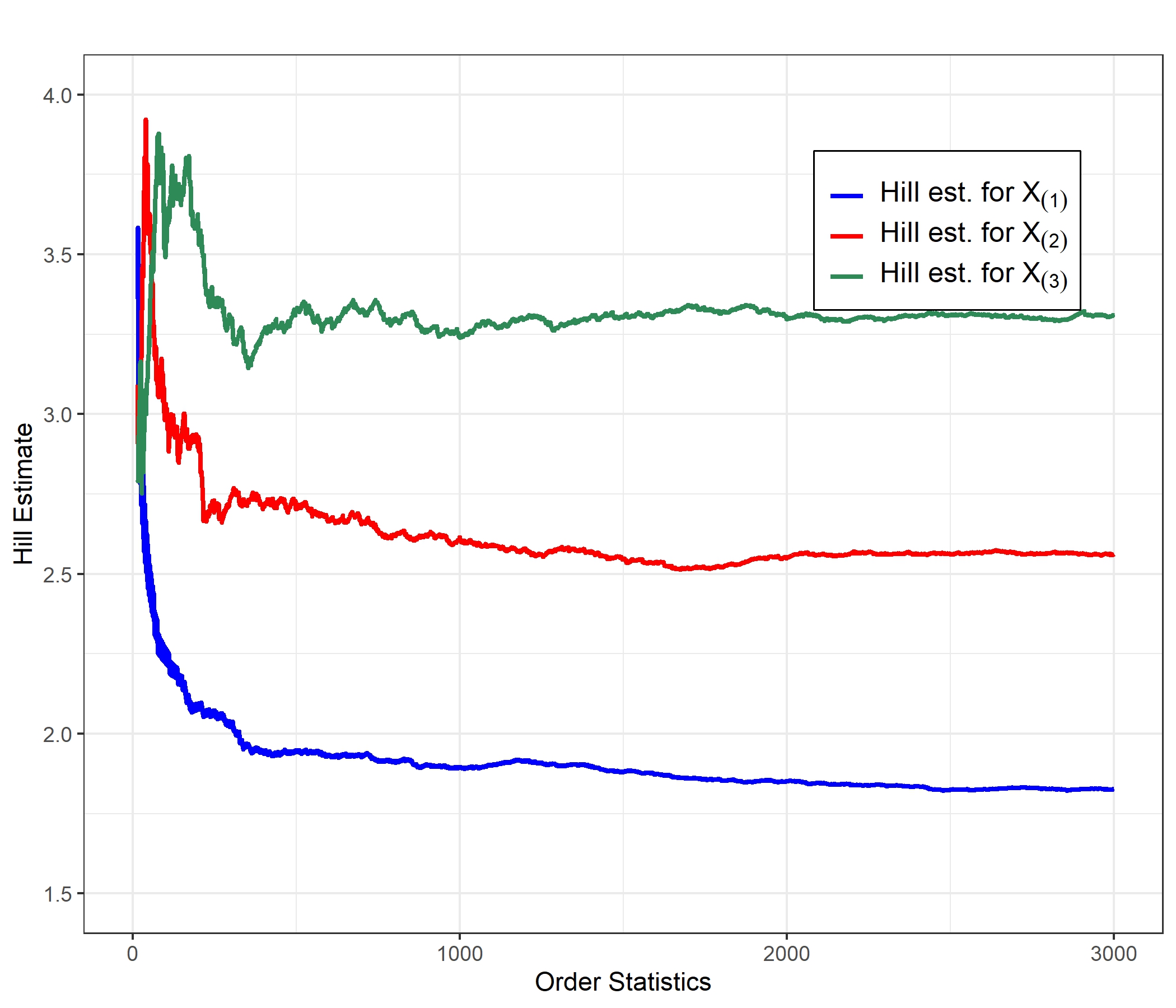}}
	\subfloat[\label{fig4:b}]{
		\includegraphics[scale=0.35]{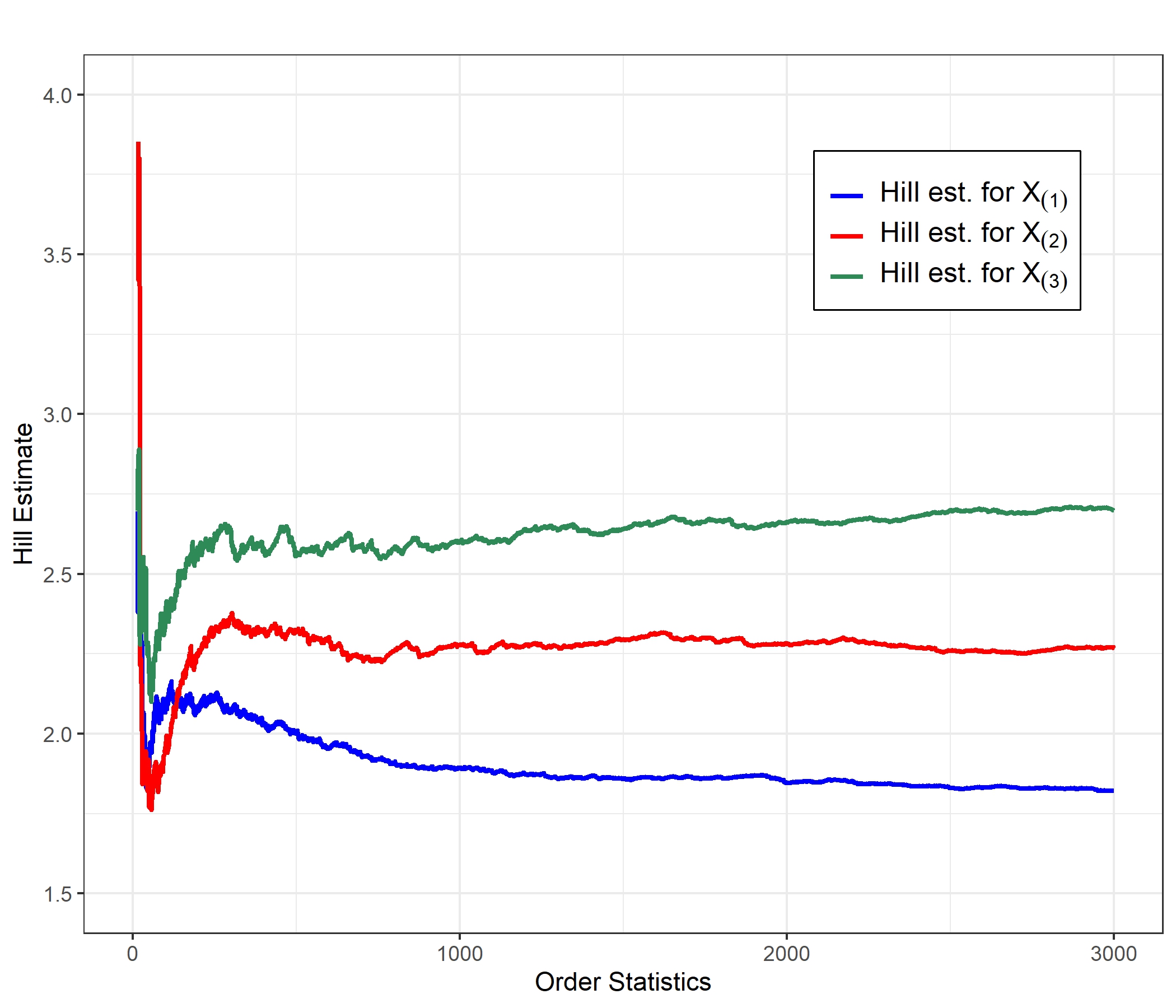}}
	\caption{Hill's plots for the Pareto risks ($\alpha=2$) with power exponential ($\gamma=3$) copula, where $\rho=0.6$ in (a) and $\rho=0.8$ in (b). The tail indices are $\alpha_1=2,\alpha_2=2.8,\alpha_3=3.2$ for the order statistics in (a) and $\alpha_1=2,\alpha_2=2.3,\alpha_3=2.5$ in (b).
 }\label{fig4}
\end{figure}

\newpage
\section{Real data applications}\label{sec: app}

We consider the Danish fire insurance data \citep{embrechts2013modelling} collected at Copenhagen Reinsurance and comprise 2167 fire losses (building, contents and profits) over the period 1980 to 1990. The losses are adjusted for inflation to reflect 1985 values and are expressed in millions of Danish Krone, and the dataset is available in R-\textit{fitdistrplus} package \citep{fitdistrplus}.

Figure \ref{fig5:a} indicates the tail indices for the building loss (left), content loss (middle) and the minimum of the two losses (right), where the estimated $\alpha$ for the individual losses are close to or less than 2, and the $\alpha_2$ for the minimum loss seems to be equal or slightly above 2. Based on these findings, the correlation between claims for building and content is estimated at a pretty high value ($\rho \ge 0.7$) for the Gaussian copula studied in \citet[Example 6.2]{Das2023inference}. However, the scatter plot (left in Figure \ref{fig5:b}) fails to suggest such a high positive association. 

As a generalization of the Gaussian copula, the elliptical copula with the Weibullian-type radius discussed in this paper 
\cl{show a clear evidence of} high dependence strength. Theorem \ref{Thm: MRV} and Example \ref{Ex: PE} suggest that $\rho = 2(\alpha/\alpha_2)^{2/\gamma} - 1$, allowing us to switch from the Gaussian copula ($\gamma=2$) to copulas with smaller $\gamma$ to achieve lower correlation values. For instance, we consider the Laplace copula ($\gamma=1$) and visualize the differences in scatter plots of simulated data. As suggested by the Hill plots, $\widehat{\alpha}=1.8$ for the building and content losses and $\widehat{\alpha}_2=2.2$ for the minimum. The estimates of correlations for the Gaussian and Laplace copulas are  $\widehat{\rho}_{G}=2\widehat{\alpha}/\widehat{\alpha}_2-1=0.64$ and $\widehat{\rho}_{L}=2(\widehat{\alpha}/\widehat{\alpha}_2)^2-1=0.36$, respectively. Simulated data, generated with Pareto ($\alpha=\widehat{\alpha}=1.8$) losses under Gaussian copula ($\rho_{G}=\widehat{\rho}_{G}$) and Laplace copula ($\rho_{LC}=\widehat{\rho}_{LC}$), with the same sample size ($n=2167$) as the real data, is then shown in Figure \ref{fig5:b}. The scatter plots show that the simulated losses generated from the Laplace copula align more closely with the real data (right in Figure \ref{fig5:b}). In contrast, the data generated from the Gaussian copula (middle in Figure \ref{fig5:b}) displays fewer observations along the axes and indicates a stronger dependence than the real data. 

Although our analysis suggests that the dependence structure in the insurance data may extend beyond the Gaussian copula to include elliptical copulas, the result only relies on visual inspection of scatter plots. For a more rigorous estimation of the parameter $\gamma$ and correlation, further investigation requires advanced statistical methods and inference procedures.

\begin{figure} [ht]
	\centering
	\subfloat[\label{fig5:a}]{
		\includegraphics[scale=0.7]{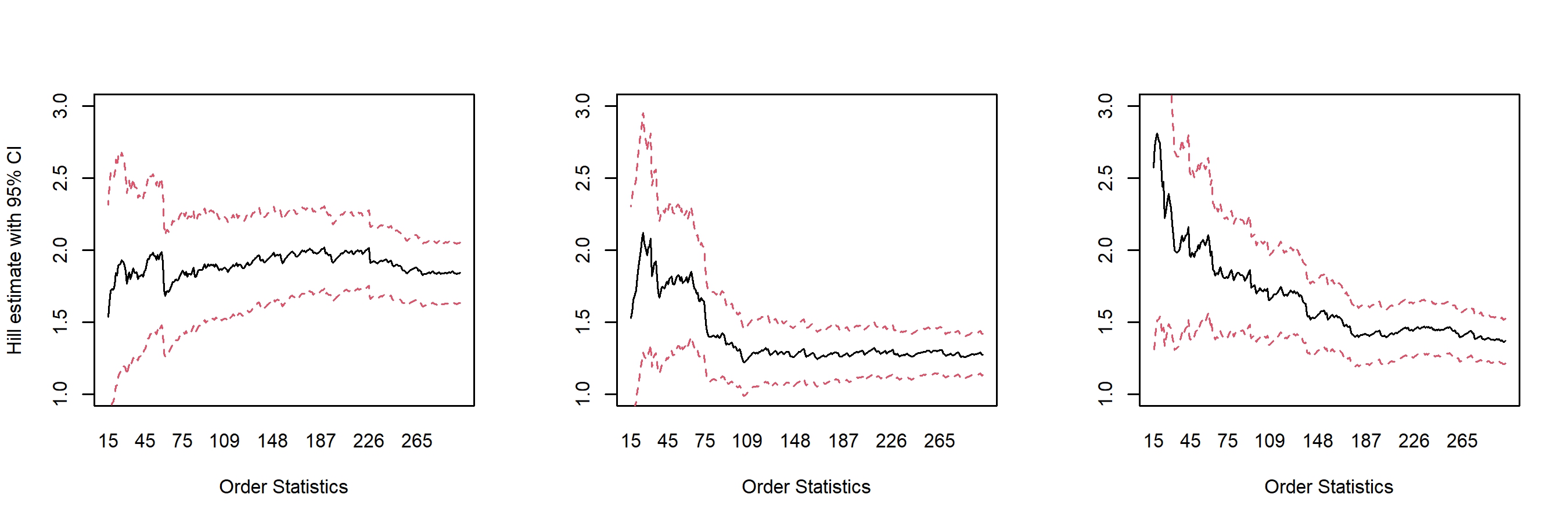}}\\
	\subfloat[\label{fig5:b}]{
		\includegraphics[scale=0.7]{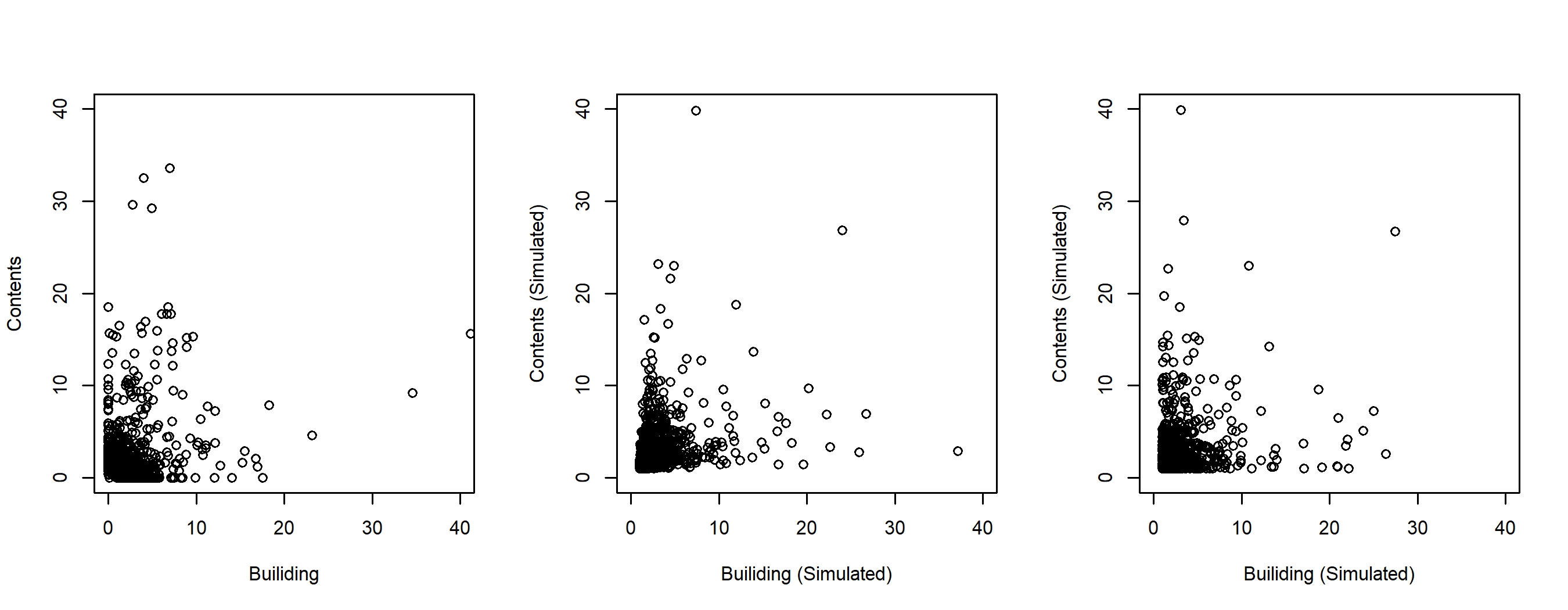}}
	\caption{Danish fire insurance data: (a) Hill estimates of tail indices of the building (left) and content (middle) and the minimum of them (right) with 95\% confidence intervals. (b) Scatter plots for real building and content losses (left), the simulated losses under Gaussian copula (middle), and the simulated losses under Laplace copula (right).
 }
\end{figure}

\section{Conclusion}\label{sec: con}
In conclusion, our study focuses on the probability of tail sets of regularly varying risks under the elliptical copula with a Weibullian-type radius. We provide the asymptotic expansion of the tail probability, demonstrating the connection between the rate of decay in the probability for joint events and the parameters in the elliptical copulas. Additionally, we characterize the multivariate regular variation on the sub-cones, provide illustrative examples and verify the theoretical result through simulations. These findings are valuable for risk measures, risk management and computing tail probabilities for various extreme tail events, with wide applications in finance, insurance, environment and biology.
Furthermore, as a generalization of the Gaussian case, we highlight the potential improvement in the fit of insurance data with elliptical copulas (e.g., Laplace copula). Further work is promising in developing statistical tools for the tail  and dependence parameter involved in the heavy tailed risk with elliptical copula, as well as the tail study with asymmetric elliptical dependence.

\section{Proofs}\label{sec: proof}
\begin{lemma}\label{Lemma: RtoZ}
Let ${\vk Z} \stackrel{d}{=} R A{\bf U}$ be an elliptical random vector with Weibullian-type $R\in \mbox{GMDA}(w,\gamma-1)$, satisfying Eq.\eqref{R} with $w(u) = \gamma L u^{\gamma-1}$. Then \cl{all margins are identically distributed satisfying $Z_1\in \mbox{GMDA}(w,\gamma-1)$}, and
\begin{eqnarray}\label{Z_j}
\pk{Z_1>u} &\sim& \frac{1}{2}\frac{\Gamma(d/2)}{\Gamma(1/2)}C \fracl{\gamma L}{2}^{-(d-1)/2} u^{\beta-\gamma(d-1)/2}\exp(-Lu^{\gamma}) \notag\\
&=:&{C'}u^{\beta'} \exp(-L'u^{\gamma'}),
\end{eqnarray}
where
\begin{eqnarray*}
C'= \frac{\Gamma(d/2)}{2\Gamma(1/2)}C \fracl{\gamma L}{2}^{-(d-1)/2}, \ \  \beta'= \beta-\gamma(d-1)/2, \ \  L' = L, \ \ \gamma' = \gamma.
\end{eqnarray*}
\end{lemma}

\begin{proof}
For the Weibullian-type radius $R\in \mbox{GMDA}(w,\gamma-1)$ with $w(u) = \gamma L u^{\gamma-1}$, and $\pk {R>u} \sim Cu^\beta \exp(-Lu^\gamma)$ for large $u$, we have $R^2\in \mbox{GMDA}(w',\gamma/2-1)$ with $w'(u) = \gamma L u^{\gamma/2-1}$ and 
\begin{eqnarray}\label{R2}
   \pk {R^2>u} \sim  Cu^{\beta/2}\exp(-Lu^{\gamma/2}). 
\end{eqnarray}
In view of \citet[Eq.(4.21)]{Hashorva2007}, the marginal risk $Z_j^2 \stackrel{d}{=} R^2  \mathcal B_{1/2,(d-1)/2}$ for all $j=1,\ldots, d$. Here $\mathcal{B}_{1/2, (d-1)/2}$ denotes the Beta random variable with parameters $1/2$ and $(d-1)/2$.
It follows further by \citet[Theorem 1.1]{debicki2018extremes} and \citet[Theorem 4.1]{HASHORVA2010} that
\begin{eqnarray*}
\pk{Z_j^2>u} &\sim& \frac{\Gamma(d/2)}{\Gamma(1/2)}\fracl{\gamma L}{2}^{-(d-1)/2}u^{-\gamma(d-1)/4}\pk{R^2>u}\\ 
&\stackrel{\eqref{R2}}\sim& \frac{\Gamma(d/2)}{\Gamma(1/2)} C \fracl{\gamma L}{2}^{-(d-1)/2} u^{-\gamma(d-1)/4+\beta/2} \exp(-Lu^{\gamma/2}),
\end{eqnarray*}
and $Z_j\in \mbox{GMDA}(w,\gamma-1)$. Noting further that $Z_j \stackrel{d}{=} RA U$ is symmetric about 0, we complete the proof with
\begin{eqnarray*}
\pk{Z_j>u} &\sim& \frac{1}{2}\frac{\Gamma(d/2)}{\Gamma(1/2)}C \fracl{\gamma L}{2}^{-(d-1)/2} u^{\beta-\gamma(d-1)/2}\exp(-Lu^{\gamma}).
\end{eqnarray*}
\end{proof}
\kai{In the following, we first present Lemma \ref{Lemma: quantile} for the asymptotic quantile function  of the elliptical marginal distribution, which will be used  to derive the transformed threshold  in Lemma \ref{Lemma: threshold}.}
\begin{lemma}\label{Lemma: quantile}
Let $\vk Z$ be an elliptical random vector defined in Lemma \ref{Lemma: RtoZ} with $\overline G_1 (u) = \pk{Z_1>u} \sim C'u^{\beta'}\exp(-Lu^\gamma), j \in \mathbb{I}$. It holds for $U(t) := 
G_1^\leftarrow(1-t^{-1})$ that
\begin{eqnarray*}
U(t) = \left(\frac{\log t}{L}\right)^{1/\gamma} +  \frac1{\gamma L}\fracl{\log t}{L}^{1/\gamma-1} \frac{\beta'}{\gamma}\left(\log (C^{'\gamma/{\beta'}}L^{-1}) +\log \log t \right) +o\left((\log t)^{1/\gamma-1}\right).
\end{eqnarray*}
\end{lemma}

\begin{proof}
Recalling $U(t) := \overline G_1^\leftarrow(t^{-1})$, we have $\overline G_1\left(U(t)\right) = 1/{t}$. Taking logarithm in both sides, we have
\begin{equation}\label{Eq.1}
  L (U(t))^\gamma  - \log C' -{\beta'} \log U(t) =  \log t,  
\end{equation}
indicating that $U(t)\to\infty$ and ${L (U(t))^\gamma} \sim {\log t}$. Taking a logarithm again on the above equation yields that 
\begin{eqnarray*}
\log L + \gamma \log U(t) - \log \log t \to 0. 
\end{eqnarray*}
Plugging $\log U(t)$ above into Eq.\eqref{Eq.1}, we have
\begin{eqnarray*}
 L (U(t))^\gamma &=& \log t  + \log C' + \frac{{\beta'}}{\gamma}(\log \log t - \log L) + o(1),
 \end{eqnarray*}
 An application of $(1+x)^a \sim 1 + a x (1+o(1))$ for $x\to0$ gives 
 \begin{eqnarray*}
U(t) &=& \left(\frac{\log t}{L}\right)^{1/\gamma} \left(1+ \frac{\log C' + \frac{{\beta'}}{\gamma}(\log \log t - \log L)}{\gamma \log t} + o\fracl{1}{\log t} \right).
\end{eqnarray*}
We complete the proof of Lemma \ref{Lemma: quantile}. 
\end{proof}

\begin{proof}[Proof of Lemma \ref{Lemma: threshold}]
For fixed $x>0$ and $\overline F_0 (t) =  \left(t^\alpha \ell(t) \right)^{-1} \in \RV_{-\alpha},$ with $\alpha>0$, applying Lemma \ref{Lemma: quantile} implies the decomposition of quantile 
\begin{eqnarray*}
z_{x, c}(t) =\overline G_1^\leftarrow(\left[(tc^{-1/\alpha}{x})^\alpha \ell(tx) \right]^{-1})= U\left((tc^{-1/\alpha}{x})^\alpha \ell(tx) \right) =: A_t+B_t+C_t,
\end{eqnarray*}
where $C_t = o\left((\log t)^{1/\gamma-1}\right)$ and 
\begin{eqnarray*}
A_t &=& \left(\frac{\alpha\log (tc^{-1/\alpha}{x})+\log \ell(tx)}{L}\right)^{1/\gamma}\\
    &=& \left(\frac{\alpha\log t}{L}\right)^{1/\gamma} \left[ 1+ \frac{(\alpha/L) \log (c^{-1/\alpha}{x}) + (1/L) \log \ell(tx)}{(\alpha/L) \log t}\right]^{1/\gamma} \\
    &=& \left(\frac{\alpha\log t}{L}\right)^{1/\gamma} \left[ 1+ \frac{1}{\gamma}\frac{(\alpha/L) \log (c^{-1/\alpha}{x}) + (1/L) \log \ell(tx)}{(\alpha/L) \log t} + o\fracl{1}{\log t}\right]\\
    &=& \left(\frac{\alpha\log t}{L}\right)^{1/\gamma} + \frac{1}{\gamma L} \frac{\alpha\log (c^{-1/\alpha}{x}) + \log \ell(t)}{ \left( \frac{\alpha \log t}{L} \right)^{1-1/\gamma}} + o \fracl{1 + \log\ell(t)}{(\log t)^{1-1/\gamma}}.
\end{eqnarray*}
In the last step, we use $\log \ell(tx) = \log \ell(t) +o(1)$. Analogously, the second term $B_t$ is given by
\begin{eqnarray*}
B_t &=& \frac1{\gamma L} \frac{{\beta'}}{\gamma}\frac{\log (C^{'\gamma/{\beta'}}L^{-1}) +\log \alpha + \log \log (tc^{-1/\alpha}{x}\ell(tx)^{1/\alpha})}{ \left[\frac{\alpha\log (tc^{-1/\alpha}{x})+\log \ell(tx)}{L}\right]^{1-1/\gamma}} \\
    &=& \frac{\beta'}{\gamma^2 L} \frac{\log (\alpha C^{'\gamma/{\beta'}}L^{-1}) + \log \log t}{ \left(\frac{\alpha\log t}{L}\right)^{1-1/\gamma}} + o\fracl{1}{(\log t)^{1-1/\gamma}}.
\end{eqnarray*}
Combining all the three parts $A_t\sim C_t$, we have
\begin{eqnarray*}
z_{x, c}(t) &=& \left(\frac{\alpha\log t}{L}\right)^{1/\gamma}  + \frac{1}{\gamma L}\frac{\log (c^{-1}x^\alpha) }{ \left( \frac{\alpha \log t}{L} \right)^{1-1/\gamma}} + \frac{1}{\gamma L} \frac{\log (C' \alpha^{{\beta'}/\gamma} L^{-{\beta'}/\gamma}) }{ \left(\frac{\alpha\log t}{L}\right)^{1-1/\gamma}}\\
&& + \frac{1}{\gamma L} \frac{\log \ell(t)}{ \left( \frac{\alpha \log t}{L} \right)^{1-1/\gamma}} +  \frac{1}{\gamma L} \frac{\log (\log t)^{{{\beta'}/\gamma}}}{ \left(\frac{\alpha\log t}{L}\right)^{1-1/\gamma}} +o \fracl{1}{(\log t)^{1-1/\gamma}}\\
&=&  \left(\frac{\alpha\log t}{L}\right)^{1/\gamma} + \frac{\log \left((c^{-1}x^\alpha C' \alpha^{{\beta'}/\gamma} L^{-{\beta'}/\gamma})^{1/\gamma L}\right) }{ \left( \frac{\alpha \log t}{L} \right)^{1-1/\gamma}} \\
&& + \frac{\log \left( [\ell(t) (\log t)^{{{\beta'}/\gamma}}]^{1/\gamma L} \right)}{ \left( \frac{\alpha \log t}{L} \right)^{1-1/\gamma}} +o \fracl{1}{(\log t)^{1-1/\gamma}}.
\end{eqnarray*}
We complete the proof of Lemma \ref{Lemma: quantile}.
\end{proof}

To show Proposition \ref{Proposition: type I}, we begin by adapting the result for type I elliptical random vectors from \citet[Theorem 3.1]{Hashorva2007} to accommodate the threshold $t\vk 1+\vk c$ and the Weibullian-type radius in Lemma \ref{Lemma: type I}.

\begin{lemma}\label{Lemma: type I}
Let $\vk Z$ be an elliptical random vector defined in Lemma \ref{Lemma: RtoZ} and $\vk e^*, \lambda, I, J$ be defined w.r.t. $\Sigma=A^\top A$ in $\mathscr{P}\left(\Sigma^{-1}\right)$ in Lemma \ref{Lemma: quadratic}.  Let $\Y$ be a Gaussian random vector in $\R^d$ with covariance
matrix $\Sigma$, and $\vk c$ be a constant in $\R^d$, then we have the asymptotic expansion $(t \to \infty)$,
\begin{eqnarray*}
\pk{\vk Z > t \vk{1} + \vk c} &=& \left(1+o(1)\right) \Upsilon (\lambda t^2+ 2t \vk c_I^\top \Sigma_{II}^{-1} \vk 1_I)^{\beta_{\gamma,I,J}} \\
&& \times t^{-|I|} \exp\left(-L(\lambda t^2+ 2t \vk c_I^\top \Sigma_{II}^{-1} \vk 1_I)^{\gamma/2}\right),
\end{eqnarray*}
where 
\begin{eqnarray*}
\Upsilon &:=& \Upsilon (\Sigma) = C{(\gamma L)^{1+|J|/2-d}}\frac{\Gamma(d/2) 2^{d/2-1} \pk{\vk Y_J> \widetilde{\vk{u}}'_J | \vk Y_I = \vk 0_I}}{(2\pi)^{|I|/2} |\Sigma_{II}|^{1/2} \prod_{i \in I} h_i },\\
h_i&:=& h_i (\Sigma) =\vk {1}_I^{\top}\Sigma_{II}^{-1} \vk e_i, \\
\lambda &:=& \lambda (\Sigma)= \vk 1_I^\top \Sigma_{II}^{-1} \vk 1_I = \min\limits_{\vk x \ge \vk 1}\vk x^\top \Sigma^{-1} \vk x,\\
\beta_{\gamma,I,J} &:=& \beta_{\gamma,I,J} (\Sigma)=\frac{\beta+|I| + \gamma(1+|J|/2-d)}{2}, \ |I|+|J| = d,\\
\widetilde{\vk{u}}'_J &:=& \widetilde{\vk{u}}'_J (\Sigma) = \lim_{t \to \infty}t^{\gamma/2} \left(\vk{1}_J-\Sigma_{JI}\Sigma^{-1}_{II} \vk{1}_I \right),
\end{eqnarray*}
and set $\pk{\vk Y_J> \widetilde{\vk{u}}'_J | \vk Y_I = \vk 0_I} = 1$ if $|I|=d$. The elements in $\widetilde{\vk{u}}'_J$ take values in $\{-\infty,0\}$, as $\Sigma_{J I}\left(\Sigma_{II}\right)^{-1} \vk 1_I = \vk e_J^*  \geq \vk 1_J$.
\end{lemma}

\begin{proof}
Applying Theorem 3.1 in \cite{Hashorva2007} with threshold $\vk t_n = t\vk 1+\vk c$ and $\pk {R>\alpha_t} \sim C\alpha_t^\beta \exp(-L\alpha_t^\gamma)$, we have
\begin{eqnarray*}
\pk{\vk Z > t \vk{1} + \vk c} &=& \left(1+o(1)\right) \frac{\Gamma(d/2) 2^{d/2-1} \pk{\vk Y_J> \widetilde{\vk{u}}'_J | \vk Y_I = \vk 0_I}}{(2\pi)^{|I|/2} |\Sigma_{II}|^{1/2} \prod_{i \in I} \vk {t}_I^{'\top}\Sigma_{II}^{-1} \vk e_i } \\
&& \times (\alpha_t \beta_t)^{1+|J|/2-d} C\alpha_t^\beta \exp(-L\alpha_t^\gamma),
\end{eqnarray*}
where
\begin{eqnarray*}
\alpha_t^2 &=& \cl{\norm{(t \vk{1} + \vk {c})_I}^2} = (t \vk{1} + \vk c)_I^\top \Sigma_{II}^{-1} (t \vk{1} + \vk c)_I\\
&=& t^2 \vk 1_I^\top \Sigma_{II}^{-1} \vk 1_I + 2t \vk c_I^\top \Sigma_{II}^{-1} \vk 1_I + \vk c_I^\top \Sigma_{II}^{-1} \vk c_I,\\
\beta_t &=& w (\alpha_t) = \gamma L\alpha_t^{\gamma-1}, \quad \vk {t}'_I = \frac{(t \vk{1} + \vk {c})_I}{\alpha_t},\\
\widetilde{\vk{u}}'_J &=& \lim_{t \to \infty} \fracl{\beta_t}{\alpha_t}^{1/2} \left((t \vk{1} + \vk {c})_J-\Sigma_{JI}\Sigma^{-1}_{II}(t \vk{1} + \vk {c})_I \right),\\
 &=& \lim_{t \to \infty}t^{\gamma/2} \left(\vk{1}_J-\Sigma_{JI}\Sigma^{-1}_{II} \vk{1}_I \right).
\end{eqnarray*}
We complete the proof using the fact $ \vk 1_I^\top \Sigma_{II}^{-1} \vk 1_I=\lambda$.
\end{proof}


Next, we replace $t$ by $u(t)$ and $\vk c$ by $\vk x^{(t)}$ in Lemma \ref{Lemma: type I} and proceed to show Proposition \ref{Proposition: type I} accordingly.

\begin{proof} [Proof of Proposition \ref{Proposition: type I}]
Recall that for constant $\vk z \in \R^d$, $u(t) \to \infty$, $\log \mathcal{L}(t) / (u(t))^{\gamma-1} \to 0$, and 
\begin{eqnarray*}
\vk x^{(t)} := \frac{\vk {z}}{(u(t))^{\gamma-1}} + \frac{\log \mathcal{L}(t)}{(u(t))^{\gamma-1}}\vk{1}+o\fracl{1}{(u(t))^{\gamma-1}},
\end{eqnarray*}
we have $\lim _{t \rightarrow \infty} \vk x^{(t)}=\mathbf{0}$. Analogously following the proof \citet[Theorem 3.1]{Hashorva2007}, 
we 
replace $t$ by $u(t)$ and $\vk c$ by $\vk x^{(t)}$ in Lemma \ref{Lemma: type I} to obtain the result. Thus,

\begin{eqnarray*}
\pk{\vk Z > u(t) \vk{1} + \vk  x^{(t)}} &=& (1+o(1)) C{(\gamma L)^{1+|J|/2-d}}\frac{\Gamma(d/2) 2^{d/2-1} \pk{\vk Y_J> \widetilde{\vk{u}}_J | \vk Y_I = \vk 0_I}}{(2\pi)^{|I|/2} |\Sigma_{II}|^{1/2} \prod_{i \in I} h_i } \\
&&\times (\lambda u^2(t)+ 2u(t) \vk  x^{(t)\top} \Sigma^{-1} \vk e^*)^{\beta_{\gamma,I,J}} (u(t))^{-|I|}\\
&&\times \exp\left(-L(\lambda u^2(t)+ 2u(t) \vk  x^{(t)\top} \Sigma^{-1} \vk e^*)^{\gamma/2}\right),
\end{eqnarray*}
where
\begin{eqnarray*}
\vk e^*_I&=&\vk 1_I,  \quad \vk e^*_J= \Sigma_{JI}(\Sigma_{II})^{-1}\vk 1_I \ge \vk 1_J,\\
\widetilde{\vk{u}}_J &=&\lim_{t \to \infty}  (\gamma L\norm{(u(t) \vk{1} + \vk { x^{(t)}})_I}^{\gamma-2})^{1/2} \left((u(t) \vk{1} + \vk { x^{(t)}})_J-\Sigma_{JI}\Sigma^{-1}_{II}(u(t) \vk{1} + \vk { x^{(t)}})_I \right).\\
&=& \lim_{t \to \infty} u(t)( \vk{1}_J-\Sigma_{JI}\Sigma^{-1}_{II}\vk{1}_I ).
\end{eqnarray*}
Note that $\widetilde{\vk{u}}_J=\widetilde{\vk{u}}'_J$ (defined in Lemma \ref{Lemma: type I}) due to the fact that $u(t) \to \infty$, and $\vk  x^{(t)\top} \Sigma^{-1} \vk e^*={\vk  x^{(t)\top}_I}\Sigma_{II}^{-1} \vk 1_I$ holds in Lemma \ref{Lemma: quadratic}. 
\end{proof}

Now, we are ready to show the tail asymptotics in Theorem \ref{Thm: tail asymptotics} using Lemmas \ref{Lemma: threshold},  \ref{Lemma: type I} and Proposition \ref{Proposition: type I}.

\begin{proof}[Proof of Theorem \ref{Thm: tail asymptotics}]
Recall that for  $A_{\vk x_S}=\left\{\vk {y} \in \R_{+}^d: y_j>x_j, \forall j \in S\right\}$ with index set $S \subseteq \{1, \ldots, d\}$, we have 
\begin{eqnarray*}
\pk{\vk X \in tA_{\vk x_S}} = \pk{X_j > tx_j, \forall j\in S}= \pk{Z_j > tz_{x_j,c_j}(t), \forall j\in S },
\end{eqnarray*}
where the elliptical random vector $\vk Z=(Z_1, \ldots,Z_d)$ has the same copula as $\vk X$. The threshold $z_{x_j,c_j}(t)=\overline G_1^\leftarrow(\overline F_j(tx_j))$ for $j\in S $ where $\overline F_j(t)=\cl{(1+o(1))}c_j(t^\alpha\ell(t))^{-1} = c_j (t^\alpha \ell(t)(1+o(1)))^{-1}\in \RV_{-\alpha}$ with $c_j, \alpha >0$. Defining 
$\vk z_S := \left(\log \left[(c_j^{-1}x_j^\alpha C' \alpha^{\beta'/\gamma} L^{-\beta'/\gamma})^{1/\gamma L}\right]\right)_{j\in S }$ and  $\mathcal{L}(t):= \cl{(1+o(1))}\left(\ell(t) (\log t)^{{\beta'/\gamma}}\right)^{1/\gamma L}$, where all the parameters are defined in Eq.\eqref{Z_j}. \cl{It follows by Lemma \ref{Lemma: threshold} that} 
\begin{eqnarray}\label{Eq: tail expansion}
&&\pk{\vk X \in tA_{\vk x_S}} \notag\\
&&= \pk{\vk Z_S > \fracl{\alpha\log t}{L}^{1/\gamma} \vk{1} + \frac{\vk z_S}{ \left( \frac{\alpha \log t}{L} \right)^{1-1/\gamma}} + \frac{\log \mathcal{L}(t)}{ \left( \frac{\alpha \log t}{L} \right)^{1-1/\gamma}}\vk{1} +o\fracl{1}{(\log t)^{1-1/\gamma}}}\notag\\
&&=: \pk{\vk Z_S > u(t) \vk{1} + \frac{\vk {z_S}}{(u(t))^{\gamma-1}} + \frac{\log \mathcal{L}(t)}{(u(t))^{\gamma-1}}\vk{1} +o\fracl{1}{(u(t))^{\gamma-1}}},
\end{eqnarray}
where $u(t) = \left((\alpha/L) \log t\right)^{1/\gamma}$.  Under the assumption of $\log \ell(t) = o\left((\log t)^{1-1/\gamma}\right)$, we have $u(t) \to \infty$ and $\log \mathcal{L}(t) / (u(t))^{\gamma-1} \to 0$. 

\kai{Define the following arguments w.r.t $\Sigma_S$ (rows and columns of $\Sigma$ with indices in $S$) as in Lemma \ref{Lemma: type I}}  
\begin{eqnarray*}
&&I_S:=I(\Sigma_S), \quad J_S:=J(\Sigma_S), \quad e^*_S:=e^*(\Sigma_S), \quad \Upsilon_S:=\Upsilon(\Sigma_S),\\
&&\beta^*:=\beta_{\gamma,I_S,J_S}=\frac{\beta+|I_S| + \gamma(1+|J_S|/2-|S|)}{2}, \quad |I_S|+|J_S| = |S|, \\
&&h_j^S:=h_j\left(\Sigma_S\right), \quad \lambda_S:=\lambda(\Sigma_S)= \min\limits_{\vk x_S \ge \vk 1_S}\vk x_S^\top \Sigma_S^{-1} \vk x_S= \vk 1_{I_S}^\top \Sigma_{I_S}^{-1} \vk 1_{I_S}, 
\end{eqnarray*} 
and applying Proposition \ref{Proposition: type I},
\begin{eqnarray*}
\pk{\vk X \in tA_{\vk x_S}} &\sim& \Upsilon_S (\lambda_S u^2(t)+ 2(u(t))^{2-\gamma} (\vk z_S+\log \mathcal{L}(t)\vk 1_S)^\top \Sigma_S^{-1} \vk e_S^*)^{\beta^*}\\
&& \times (u(t))^{-|I_S|}\exp\left(-L(\lambda_S u^2(t)+ 2(u(t))^{2-\gamma} (\vk z_S+\log \mathcal{L}(t)\vk 1_S)^\top \Sigma_S^{-1} \vk e_S^*)^{\gamma/2}\right).
\end{eqnarray*}
For $\gamma > 0$, by the Taylor expansion of
\begin{eqnarray*}
&&(\lambda_S u^2(t)+ 2(u(t))^{2-\gamma} (\vk z_S+\log \mathcal{L}(t)\vk 1_S)^\top \Sigma_S^{-1} \vk e_S^*)^{\beta^*} \\
&& \quad \quad \sim (\lambda_Su^2(t))^{\beta^*}\left( 1 + \frac{2\beta^*}{\lambda_S u^\gamma(t)} (\vk z_S+\log \mathcal{L}(t)\vk 1_S)^\top \Sigma_S^{-1} \vk e_S^*)\right),
\end{eqnarray*}
we have
\begin{eqnarray*}
\pk{\vk X \in tA_{\vk x_S}}
&\sim& \Upsilon_S (\lambda_Su^2(t))^{\beta^*}\left( 1 + \frac{2\beta^*}{\lambda_S u^\gamma(t)} (\vk z_S+\log \mathcal{L}(t)\vk 1_S)^\top \Sigma_S^{-1} \vk e_S^*)\right)\\
&& \times (u(t))^{-|I_S|}\exp\left(-L \left(\lambda_S^{\gamma/2} u^{\gamma}(t) +\gamma\lambda_S^{\gamma/2-1}(\vk z_S+\log \mathcal{L}(t)\vk 1_S)^\top \Sigma_S^{-1} \vk e_S^*   \right)\right)\\
&=& \Upsilon_S \lambda_S^{\beta^*} (u(t))^{2\beta^*}\left( 1 + \frac{2\beta^*}{\lambda_S u^\gamma(t)} (\vk z_S+\log \mathcal{L}(t)\vk 1_S)^\top \Sigma_S^{-1} \vk e_S^*)\right)\\
&& \times (u(t))^{-|I_S|} \mathcal{L}(t)^{-L\gamma \lambda_S^{\gamma/2}} \ t^{-\alpha\lambda_S^{\gamma/2}}\\
&& \times \prod_{j \in I_S}\exp\left(- h_j^S L\gamma \lambda_S^{\gamma/2-1} \log [(c_j^{-1}x_j^\alpha C' \alpha^{\beta'/\gamma} L^{-\beta'/\gamma})^{1/\gamma L}]  \right)\\
&=& \Upsilon_S \lambda_S^{\beta^*} \fracl{\alpha \log t}{L}^{\frac{2\beta^*}{\gamma}}\\
&& \times \left(1  +\frac{2\beta^*\sum_{j \in I_S} h_j^S \log (c_j^{-1}x_j^\alpha C' \alpha^{\frac{\beta'}{\gamma}} L^{-\frac{\beta'}{\gamma}}) }{\alpha \gamma \lambda_S \log t} + \frac{2\beta^* \log \left(\ell(t)(\log t)^{\frac{\beta'}{\gamma}}\right) }{\alpha \gamma \log t} \right)\\
&& \times \fracl{\alpha}{L}^{-\frac{|I_S|}{\gamma}} (\log t)^{-\frac{|I_S|}{\gamma}-\frac{\beta'}{\gamma}\lambda_S^{\gamma/2}} (t^\alpha \ell(t))^{-\lambda_S^{\gamma/2}}\\
&& \times (C' \alpha^{\beta'/\gamma} L^{-\beta'/\gamma})^{-\lambda_S^{\gamma/2}} \prod_{j \in I_S} \left(\frac{x_j^\alpha}{c_j}\right)^{- h_j^S\lambda_S^{\gamma/2-1}}. 
\end{eqnarray*}
Dropping the second and third terms in
\begin{eqnarray*}
1  +\frac{2\beta^*\sum_{j \in I_S} h_j^S \log (c_j^{-1}x_j^\alpha C' \alpha^{\frac{\beta'}{\gamma}} L^{-\frac{\beta'}{\gamma}}) }{\alpha \gamma \lambda_S \log t} + \frac{2\beta^* \log \left(\ell(t)(\log t)^{\frac{\beta'}{\gamma}}\right) }{\alpha \gamma \log t} ,
\end{eqnarray*}
we have
\begin{eqnarray*}
\pk{\vk X \in tA_{\vk x_S}}&\sim& \Upsilon_S \lambda_S^{\beta_{\gamma,I_S,J_S}} \fracl{\alpha}{L}^{-\frac{|I_S|-2\beta_{\gamma,I_S,J_S}}{\gamma}} (\log t)^{-\frac{|I_S|}{\gamma}-\frac{\beta'}{\gamma}\lambda_S^{\gamma/2}+ \frac{2\beta_{\gamma,I_S,J_S}}{\gamma}} \\
&& \times (t^\alpha \ell(t))^{-\lambda_S^{\gamma/2}} (C' \alpha^{\beta'/\gamma} L^{-\beta'/\gamma})^{-\lambda_S^{\gamma/2}} \prod_{j \in I_S} (c_j^{-1}x_j^\alpha)^{- h_j^S\lambda_S^{\gamma/2-1}}   \\
&=& \Upsilon_S \lambda_S^{\beta_{\gamma,I_S,J_S}} \fracl{\alpha}{L}^{\frac{\beta}{\gamma}+1-\frac{|I_S|}{2}-\frac{|S|}{2}} (\log t)^{\frac{\beta}{\gamma}+1-\frac{|I_S|}{2}-\frac{|S|}{2}-\frac{\beta'}{\gamma}\lambda_S^{\gamma/2}} \\
&& \times (t^\alpha \ell(t))^{-\lambda_S^{\gamma/2}} (C' \alpha^{\beta'/\gamma} L^{-\beta'/\gamma})^{-\lambda_S^{\gamma/2}} \prod_{j \in I_S} (c_j^{-1}x_j^\alpha)^{- h_j^S\lambda_S^{\gamma/2-1}}.
\end{eqnarray*}
\cl{Consequently, Theorem \ref{Thm: tail asymptotics} is obtained. }
\end{proof}
\kai{Next, we will present the proof of Theorem \ref{Thm: MRV} by utilizing Theorem \ref{Thm: tail asymptotics}.}
\begin{proof}[Proof of Theorem \ref{Thm: MRV}]
For any $\vk x \in \R_+^d$, $\nu_1\left(\partial [\vk 0, \vk x]^c \right)=0$ always hold. By the inclusion-exclusion principle and tail equivalence (with index $\alpha_1$) for all margins of $\vk X$,
\begin{eqnarray*}
\lim\limits_{t \to \infty} t\pk{\frac{\vk X}{b_1(t)}\in [\vk 0, \vk x]^c} &\le& \sum\limits_{j=1}^d \lim\limits_{t \to \infty} t \pk{X_j > b_1(t)x_j} =\sum\limits_{j=1}^d x_j^{-\alpha_1},\\
\lim\limits_{t \to \infty} t\pk{\frac{\vk X}{b_1(t)}\in [\vk 0, \vk x]^c} &\ge& \sum\limits_{j=1}^d \lim\limits_{t \to \infty} t \pk{X_j > b_1(t)x_j} \\
&&-\sum\limits_{j,k=1; j\ne k}^d \lim\limits_{t \to \infty} t \pk{X_j > b_1(t)x_j, X_k > b_1(t)x_k}.
\end{eqnarray*}
a) Since $\pk{X_j > b_1(t)x_j, X_k > b_1(t)x_k}=o(t^{-1})$ for covariance $|\rho_{jk}|<1$ in matrix $\Sigma=AA^\top$, according to tail independence property for the type I elliptical copula \citep{Schmidt2002, FRAHM2003}, the risk $\vk X \in \MRV(\alpha_1,b_1,\nu_1,\EE_d^{(1)})$ and
\begin{eqnarray*}
\lim\limits_{t \to \infty} t\pk{\frac{\vk X}{b_1(t)}\in [\vk 0, \vk x]^c} =\sum\limits_{j=1}^d x_j^{-\alpha_1}=\nu_1\left([\vk 0, \vk x]^c\right).
\end{eqnarray*}

b) For $2\le i \le d$, by Lemma \ref{Lemma: cone}, it suffices to show 
\begin{eqnarray*}
\lim\limits_{t \to \infty} t\pk{\frac{\vk X}{b_i(t)}\in A_{\vk x_S}}= \lim\limits_{t \to \infty} b_i^{\leftarrow}(t)\pk{\vk X \in tA_{\vk x_S}} = \nu_i(A_{\vk x_S}).
\end{eqnarray*}
Using Theorem \ref{Thm: tail asymptotics}, we have the asymptotic tail probability
\begin{eqnarray*}
\pk{\vk X \in tA_{\vk x_S}}&\sim& \Upsilon_S \lambda_S^{\beta_{\gamma,I_S,J_S}} \fracl{\alpha}{L}^{\frac{\beta}{\gamma}+1-\frac{|I_S|}{2}-\frac{|S|}{2}} (\log t)^{\frac{\beta}{\gamma}+1-\frac{|I_S|}{2}-\frac{|S|}{2}-\frac{\beta'}{\gamma}\lambda_S^{\gamma/2}} \\
&& \times (b^{\leftarrow}(t))^{-\lambda_S^{\gamma/2}} (C' \alpha^{\beta'/\gamma} L^{-\beta'/\gamma})^{-\lambda_S^{\gamma/2}} \prod_{j \in I_S} (c_j^{-1}x_j^\alpha)^{- h_j^S\lambda_S^{\gamma/2-1}}   \\
&\stackrel{\eqref{Z_j}}\sim& \Upsilon_S \lambda_S^{\beta_{\gamma,I_S,J_S}} \fracl{\alpha}{L}^{\frac{\beta'}{\gamma}+\frac{1}{2}-\frac{|I_S|}{2}} (\log t)^{-\frac{|I_S|}{2}+\frac{1}{2} + \frac{\beta'}{\gamma}(1-\lambda_S^{\gamma/2})} \\
&& \times (b^{\leftarrow}(t))^{-\lambda_S^{\gamma/2}} (C' \alpha^{\beta'/\gamma} L^{-\beta'/\gamma})^{-\lambda_S^{\gamma/2}} \prod_{j \in I_S} (c_j^{-1}x_j^\alpha)^{- h_j^S\lambda_S^{\gamma/2-1}} \\
&\sim& \Upsilon_S \lambda_S^{\beta_{\gamma,I_S,J_S}} C^{'-\lambda_S^{\gamma/2}} \fracl{\alpha\log t}{L}^{-\frac{|I_S|}{2}+\frac{1}{2} + \frac{\beta'}{\gamma}(1-\lambda_S^{\gamma/2})} \\
&& \times (b^{\leftarrow}(t))^{-\lambda_S^{\gamma/2}}  \prod_{j \in I_S} (c_j^{-1}x_j^\alpha)^{- h_j^S\lambda_S^{\gamma/2-1}}. 
\end{eqnarray*}
Note that
\begin{eqnarray*}
\mathcal{S}_i &:=& \left\{ S \subset \mathbb{I}: |S|\ge i, \vk 1_{I_S}^{\top} \Sigma_{I_S}^{-1}\vk 1_{I_S}= \min\limits_{\widetilde{S} \subset \mathbb{I},|\widetilde{S}| \ge i } \vk 1_{I_{\widetilde{S}}}^{\top} \Sigma^{-1}_{I_{\widetilde{S}}} \vk 1_{I_{\widetilde{S}}} \right\},\\
I_i &:=& \arg\min\limits_{S \in \mathcal{S}_i} |I_S|,
\end{eqnarray*}
implying that, if $S \notin \mathcal{S}_i$, then $\lambda_i < \lambda_S$, while if $S \in \mathcal{S}_i$, then  $|I_i| \le |I_S|$.
Hence if $S \notin \mathcal{S}_i$, then $\lambda_i^{\gamma/2} < \lambda_S^{\gamma/2}$ ($\gamma>0$), and
\begin{eqnarray*}
b_i^{\leftarrow}(t)\pk{\vk X \in tA_{\vk x_S}} &\sim& C^{'\lambda_i^{\gamma/2}-\lambda_S^{\gamma/2}} \Upsilon_S \lambda_S^{\beta_{\gamma,I_S,J_S}} \prod_{j \in I_S} (c_j^{-1}x_j^\alpha)^{- h_j^S\lambda_S^{\gamma/2-1}}  \\
&& \times \kai{(t^\alpha \ell(t))^{\lambda_i^{\gamma/2}-\lambda_S^{\gamma/2}}} \fracl{\alpha\log t}{L}^{\frac{|I_i|-|I_S|}{2}+\frac{\beta'}{\gamma}(\lambda_i^{\gamma/2}-\lambda_S^{\gamma/2})},
\end{eqnarray*}
where $(t^\alpha \ell(t))^{\lambda_i^{\gamma/2}-\lambda_S^{\gamma/2}}$ dominates the tail behaviour and $g_i(t)=b_i^{\leftarrow}(t)\pk{\vk X \in tA_{\vk x_S}} \to 0$, i.e., $g_i(t)\in \RV_{-a}$ with $a=\alpha(\lambda_S^{\gamma/2}-\lambda_i^{\gamma/2})>0$.

If $S \in \mathcal{S}_i$ and $|I_i| < |I_S|$, then $\lambda_i^{\gamma/2} = \lambda_S^{\gamma/2}$, and
\begin{eqnarray*}
g_i(t)=b_i^{\leftarrow}(t)\pk{\vk X \in tA_{\vk x_S}} &\sim& \Upsilon_S \lambda_S^{\beta_{\gamma,I_S,J_S}} \prod_{j \in I_S} (c_j^{-1}x_j^\alpha)^{- h_j^S\lambda_S^{\gamma/2-1}} \fracl{\alpha\log t}{L}^{\frac{|I_i|-|I_S|}{2}}\to 0.
\end{eqnarray*}

If $S \in \mathcal{S}_i$ and $|I_i| = |I_S|$, then $\lambda_i^{\gamma/2} = \lambda_S^{\gamma/2}$, and
\begin{eqnarray*}
\lim\limits_{t \to \infty} b_i^{\leftarrow}(t)\pk{\vk X \in tA_{\vk x_S}} &=& \Upsilon_S \lambda_S^{\beta_{\gamma,I_S,J_S}} \prod_{j \in I_S} (c_j^{-1}x_j^\alpha)^{- h_j^S\lambda_S^{\gamma/2-1}} = \nu_i(A_{\vk x_S}) .
\end{eqnarray*}
Consequently, the claim in Theorem \ref{Thm: MRV} follows.
\end{proof}

The Mill's ratio for Weibullian type radius is derived analogously from the traditional method for Gaussian distribution.
\begin{proof}[Proof of Lemma \ref{Lemma: MR for R}]
Recalling the density $h(r)= Cr^{d-1}\exp\left(-{{r}^{2\kappa}}/{2}\right),\, r>0, \kappa>0$, we have
\begin{eqnarray*}
  \int_{r}^{\infty} h(u) du \le \int_{r}^{\infty} \frac{u^{2\kappa-d}}{r^{2\kappa-d}}Cu^{d-1}\exp\left(-\frac{{u}^{2\kappa}}{2}\right)du =  -\frac{C}{\kappa r^{2\kappa-d}}\exp\left(-\frac{{u}^{2\kappa}}{2}\right)\bigg|_{r}^{\infty}= \frac{h(r)}{\kappa r^{2\kappa-1}}.
\end{eqnarray*}
Noting further that $h'(r)=((d-1)r^{-1}-\kappa r^{2\kappa-1})h(r)$, we have
\begin{eqnarray*}
  \left(\kappa+\frac{2\kappa}{r^{2\kappa}} \right)\int_{r}^{\infty} h(u)du &\ge& \int_{r}^{\infty} \left(\kappa+\frac{2\kappa}{u^{2\kappa}} \right) h(u)du \\
  &\ge& \int_{r}^{\infty} \left(\kappa+\frac{2\kappa-d}{u^{2\kappa}} \right) h(u)du\\
  &=& \int_{r}^{\infty} \frac{(\kappa u^{2\kappa-1}-(d-1)u^{-1})h(u) u^{2\kappa-1}}{u^{4\kappa-2}}+\frac{(2\kappa-1) u^{2\kappa-2}}{u^{4\kappa-2}} h(u) du\\
  &=& -\frac{h(u)}{u^{2\kappa-1}}\bigg|_{r}^{\infty}= \frac{h(r)}{r^{2\kappa-1}}.
\end{eqnarray*}
Thus,
\begin{eqnarray*}
   \int_{r}^{\infty} h(u)du \ge \frac{r}{\kappa(r^{2\kappa}+2)} h(r),
\end{eqnarray*}
which completes the proof.
\end{proof}

\section{Declaration}
\subsection*{Declaration of competing risk}
The authors declare that they have no known competing financial interests.

\subsection*{Acknowledgements}
This work was partially supported by the Research Development Fund [RDF1912017], and the Post-graduate Research Fund [PGRS2112022], Xi'an Jiaotong-Liverpool University. 
We thank Prof. Enkelejd Hashorva for his insightful suggestions and comments for this work.

\bibliographystyle{abbrvnat}
\bibliography{reference} 

\end{document}